\documentclass[UTF8,10pt]{article}
	
	\usepackage[left=3.18cm,right=3.18cm,top=2.54cm,bottom=2.54cm]{geometry}
	\usepackage{amsfonts}
	\usepackage{amsmath}
	\usepackage{amssymb}
	\usepackage{pgfplots}
	\usepackage{tikz}
	\usetikzlibrary{arrows}
	\usepackage{titlesec}
	\usepackage{bm}
	\usepackage{appendix}
	\usepackage{tikz-cd}
	\usepackage{mathtools}
	\usepackage{amsthm}
	\usepackage{setspace}
	\usepackage{enumitem}
	\setenumerate[1]{itemsep=0pt,partopsep=0pt,parsep=\parskip,topsep=5pt}
	\setitemize[1]{itemsep=0pt,partopsep=0pt,parsep=\parskip,topsep=5pt}
	\setdescription{itemsep=0pt,partopsep=0pt,parsep=\parskip,topsep=5pt}
	\usepackage[colorlinks,linkcolor=black,anchorcolor=black,citecolor=black]{hyperref}
	\usepackage[numbers]{natbib}
	\usepackage{cases}
	\usepackage{fancyhdr}
	\usepackage[scaled=0.92]{helvet}	
	\usepackage{upgreek}
	\usepackage{txfonts}
	\usepackage{verbatim}

	\theoremstyle{definition}
	\newtheorem{thm}{Theorem}[section]

	\newtheorem{lem}[thm]{Lemma}
	\newtheorem{prop}[thm]{Proposition}
	
	\newtheorem{cor}[thm]{Corollary}
	
	\newtheorem{rmk}{Remark}[section]

	\newcommand{\R}{\mathbb{R}}  
	\newcommand{\Z}{\mathbb{Z}}
	
	\newcommand{\N}{\mathbb{N}}
	
	\newcommand{\T}{\mathbb{T}}
	\newcommand{\TT}{\mathcal{T}}
	\newcommand{\TP}{\overline{\p}{}}

	\newcommand{\bd}[1]{\mathbf{#1}}  
	\newcommand{\RR}{\mathcal{R}}      
	\newcommand{\Om}{\Omega}
	\newcommand{\q}{\quad}
	\newcommand{\p}{\partial}

	\newcommand{\dd}{\mathrm{d}}

	\newcommand{\nab}{\nabla}

	\newcommand{\lap}{\Delta}
	\newcommand{\no}{\nonumber}

	\newcommand{\lleq}{\stackrel{L}{=}}

	\newcommand{\dx}{\,\mathrm{d}x}

	\newcommand{\dt}{\,\mathrm{d}t}
	
	\newcommand{\dtau}{\,\mathrm{d}\tau}

	\newcommand{\lee}{\langle}
	\newcommand{\ree}{\rangle}

	\newcommand{\UU}{\mathbf{U}}
	\newcommand{\BB}{\mathbf{B}}
	\newcommand{\QQ}{\mathbf{Q}}

	\newcommand{\qc}{\check{q}}

	\newcommand{\io}{\int_{\Omega}}

	\newcommand{\ddt}{\frac{\mathrm{d}}{\mathrm{d}t}}
	\numberwithin{equation}{section}
	
	\newcommand{\eps}{\varepsilon}
	
	\newcommand{\bp}{(B\cdot\nab)}
	\newcommand{\rbp}{(\rho^{-1}B\cdot\nab)}

	\usepackage{xcolor}

	\allowdisplaybreaks
	
\begin{document}
		\title{\bf Uniform Anisotropic Regularity and Low Mach Number Limit of Non-isentropic Ideal MHD Equations \\ with a Perfectly Conducting Boundary}
		\date{\today}
		\author{{\sc Qiangchang JU}\thanks{Institute of Applied Physics and Computational Mathematics, Beijing, P.R. China.
				Email: \texttt{ju\_qiangchang@iapcm.ac.cn}}\,\,\,\,\,\,{\sc Jiawei WANG}\thanks{Hua Loo-Keng Center for Mathematical Sciences, Academy of Mathematics and Systems Science, Chinese Academy of Sciences, Beijing, P.R. China.
				Email: \texttt{wangjiawei@amss.ac.cn}}\,\,\,\,\,\, {\sc Junyan ZHANG}\thanks{Department of Mathematics, National University of Singapore, Singapore. 
				Email: \texttt{zhangjy@nus.edu.sg}} }
		\maketitle
		
		\begin{abstract}
			We prove the low Mach number limit of non-isentropic ideal magnetohydrodynamic (MHD) equations with {\it general initial data} in the half-space whose boundary satisfies the perfectly conducting wall condition. By observing a special structure contributed by Lorentz force in vorticity analysis, we establish uniform estimates in suitable anisotropic Sobolev spaces with weights of Mach number determined by the number of material derivatives. We also observe that the entropy has the enhanced regularity in the direction of the magnetic field. These two observations help us get rid of the loss of derivatives and weights of Mach number in vorticity analysis caused by the simultaneous appearance of entropy, general initial data and the magnetic field, which is one of the major difficulties that do not appear in Euler equations or the isentropic problems. By utilizing the technique of Alinhac good unknowns, the anti-symmetric structure is preserved in the tangential estimates for the system differenetiated by high-order material derivatives.

		\end{abstract}
		
		
		\setcounter{tocdepth}{1}
		\tableofcontents
		
		\section{Introduction}
		
		In this paper, we consider the scaled non-isentropic compressible ideal magnetohydrodynamic (MHD) equations 
		\begin{equation}
			\begin{cases}
				D_t\rho+\rho(\nabla\cdot u)=0~~~&\text{in}~[0,T]\times\Om, \\
				\rho D_t u=B\cdot\nabla B-\nabla P,~~P:=\varepsilon^{-2}p+\frac{1}{2}|B|^2~~~& \text{in}~[0,T]\times\Om, \\
				D_t B=B\cdot\nabla u-B(\nabla\cdot u)~~~&\text{in}~[0,T]\times\Om, \\
				\nabla\cdot B=0~~~&\text{in}~[0,T]\times\Om,\\
				D_tS=0~~~&\text{in}~[0,T]\times\Om,
			\end{cases}\label{CMHD}
		\end{equation}
		describing the motion of a compressible conducting fluid in an electro-magnetic field. Here we set $\Om=\mathbb{R}^d_+:=\{x\in \mathbb{R}^d: x_d>0\}$ for $d=2,3$ with boundary $\Sigma:=\{x_d=0\}$. $\nabla:=(\p_1,\cdots,\p_d)^\mathrm{T}$ is the standard spatial derivative. $D_t:=\p_t+u\cdot\nabla$ is the material derivative. The fluid velocity, the magnetic field, the fluid density, the fluid pressure and the entropy are denoted by $u=(u_1,\cdots,u_d)^\mathrm{T}$, $B=(B_1,\cdots,B_d)^\mathrm{T}$, $\rho$, $p$ and $S$ respectively. Note that the last equation of \eqref{CMHD} is derived from the equation of total energy and Gibbs relation. Defined as the ratio of characteristic fluid velocity to the sound speed, the Mach number $\eps$ is a {\it dimensionless} parameter that measures the compressibility of the fluid. Throughout this manuscript, Einstein's summation convention is adopted and repeated indices range from 1 to $d$.
		
		We assume the fluid density $\rho=\rho(p,S)>0$ to be a given smooth function of $p$ and $S$ which satisfies 
		\begin{align}
			\rho\geq \bar{\rho_0}>0,~~~\frac{\p \rho}{\p p}> 0, \q \text{in}\,\, \bar{\Om},
			\label{EoS}
		\end{align}for some positive constant $\bar{\rho_0}$. For instance, we have ideal fluids $\rho(p,S)=p^{1/\gamma}e^{-S/\gamma}$ with $\gamma>1$ for a polytropic gas. These two conditions also guarantee the hyperbolicity of system \eqref{CMHD}.

		The initial and boundary conditions of system \eqref{CMHD} are
		\begin{align}
			\label{data}	    (p,u,B,S)|_{t=0}=(p_0,u_0, B_0,S_0) ~~~& \text{in }[0,T]\times\Om,\\
			\label{bdry cond}			u_d=0,~~B_d=0 ~~~& \text{on }[0,T]\times\Sigma,
		\end{align} where the boundary condition for $u_d$ is the slip boundary condition, and the boundary condition for $B_d$ shows that $\Sigma$ is a perfectly conducting wall. The meaning of the second condition is that the plasma is closed off from the outside world by a perfectly conducting wall. As stated in \cite[Chapter 4.6]{MHDphy}, this model is appropriate for the study of equilibrium, waves and instabilities of confined plasma as used in thermonuclear research and is also served as the simplest, the most relevant model to describe confined plasmas. Such configurations refer to tokamaks.

		\begin{rmk} The conditions $\nabla\cdot B=0$ in $\Om$ and $B_d=0$ on $\Sigma$ are both constraints for initial data so that the MHD system is not over-determined. One can show that they propagate within the lifespan of the solution. Using the theory of hyperbolic system with charateristic boundary conditions \cite{Rauch1985}, one can show that the correct number of boundary conditions is 1. So, $B_d|_{\Sigma}=0$ has to be an initial constraint.
		\end{rmk}
		
		To make the initial-boundary-value problem \eqref{CMHD}-\eqref{bdry cond} solvable, we need to require the initial data satisfying the compatibility conditions up to certain order. For $m\in\N$, we define the $m$-th order compatibility conditions to be
		\begin{equation}\label{comp cond}
		B_d|_{t=0}=0\text{ and }\p_t^j u_d|_{t=0}=0\q\text{ on }\Sigma,~~0\leq j\leq m.
		\end{equation}It should be noted that \eqref{comp cond} indicates $\p_t^j B_d|_{t=0}$ on $\Sigma$ for $0\leq j\leq m$ and we refer to Trakhinin \cite[Section 4.1]{Trakhinin2008CMHDVS} for the proof.

		Let $a:=\frac{1}{\rho}\frac{\p\rho}{\p p}$. Since $\frac{\p \rho}{\p p}>0$ implies $a(p,S)>0$, in view of $D_t S=0$, the first equation of \eqref{CMHD} is equivalent to 
		\begin{equation} \label{continuity eq f}
			a D_t p  +\nab\cdot u=0. 
		\end{equation}
		Thus the compressible MHD system is now reformulated as follows
		\begin{equation}\label{CMHD2}
			\begin{cases}
				a D_t p  +\nab\cdot u=0~~~&\text{in}~[0,T]\times\Om, \\
				\rho D_t u=B\cdot\nabla B-\nabla P,~~P:=\eps^{-2}p+\frac{1}{2}|B|^2~~~& \text{in}~[0,T]\times\Om, \\
				D_t B=B\cdot\nabla u-B(\nabla\cdot u)~~~&\text{in}~[0,T]\times\Om, \\
				\nabla\cdot B=0~~~&\text{in}~[0,T]\times\Om,\\
				D_t S=0~~~&\text{in}~[0,T]\times\Om,\\
				a=a(p,S)>0,\quad\rho=\rho(p, S)>0~~~&\text{in}~[0,T]\times\bar{\Om}.\\
				u_d=B_d=0 ~~~& \text{on}~[0,T]\times\Sigma,\\
				(p,u,B,S)|_{t=0}=(p_0,u_0, B_0, S_0)~~~& \text{on}~\{t=0\}\times\Om.
			\end{cases}
		\end{equation}

		When considering the incompressible limit, that is, when $\eps>0$ is sufficiently small, it is more convenient to symmetrize the compressible MHD system by using the transformation $p=1+\eps q$, yielding
		\begin{equation}\label{CMHD3}
			\begin{cases}
				a D_t q  + \eps^{-1} \nab\cdot u=0~~~&\text{in}~[0,T]\times\Om, \\
				\rho D_t u+\eps^{-1}\nab q+\frac{1}{2}\nab|B|^2-B\cdot\nabla B=0~~~& \text{in}~[0,T]\times\Om, \\
				D_t B=B\cdot\nabla u-B(\nabla\cdot u)~~~&\text{in}~[0,T]\times\Om, \\
				\nabla\cdot B=0~~~&\text{in}~[0,T]\times\Om,\\
				D_t S=0~~~&\text{in}~[0,T]\times\Om,\\
				a=a(\eps q,S)>0,\quad \rho=\rho(\eps q,S)>0~~~&\text{in}~[0,T]\times\bar{\Om}.\\
				u_d=B_d=0 ~~~& \text{on}~[0,T]\times\Sigma,\\
				(q,u,B,S)|_{t=0}=(q_0,u_0, B_0, S_0)~~~& \text{on}~\{t=0\}\times\Om.\\
			\end{cases}
		\end{equation}

		\subsection{An overview of previous results}\label{sect history}
		As is well-known in fluid dynamics, see for example in Majda's book \cite[Chapter 2.4]{Majdalimit}, we can formally derive the incompressible fluid equations from the compressible ones, which corresponds to passing to the limit in certain dimensionless form as the Mach number goes to zero. In particular, for inviscid fluids, taking low Mach number limit is a singular limit process of a hyperbolic system with large parameter, for example the coefficients $\eps^{-1}$ in \eqref{CMHD3}. To study such singular limit, we shall classify the initial data to be two types.
		\begin{itemize}
		\item Well-prepared initial data: $\nab \cdot u_0 = O(\eps^k),~\nab q_0 = O(\eps^k)$ for $k\geq 1$. In such case, the compressible data is exactly a slight perturbation of a given incompressible data. Uniform estimates in low Mach regime immedately give the strong convergence thanks to the uniform boundedness of first-order time derivatives.
		\item General initial data: $\nab\cdot u_0 = O(1),~\nab q_0 = O(1)$. In such case, the compressible data includes a large perturbation which is actually a highly oscillatory acoustic wave that propagates with a speed of $O(1/\eps)$.  One has to filter such acoustic wave and find suitable function spaces for the {\it strong} convergence.
		\end{itemize}
		The low Mach number limit of Euler equations in $\R^d,\T^d$ or a fixed domain has been studied extensively. 
		\begin{center}
			\begin{tabular}{|c|c|c|}
			\hline Initial data & Isentropic Euler & Non-isentropic Euler \\
			\hline \begin{tabular}{c} 
			Well-prepared 
			\end{tabular} & \begin{tabular}{c} 
			Klainerman-Majda \cite{Klainerman1981limit,Klainerman1982limit} \\
			Ebin \cite{Ebin1982limit} 
			\end{tabular} & Schochet \cite{Schochet1986limit, Schochet1987limit} \\
			\hline General & \begin{tabular}{c} 
			Ukai \cite{Ukai1986limit}, Asano \cite{Asano1987limit}, Isozaki \cite{Isozaki1987limit}, Iguchi \cite{Iguchi1997limit}, \\
			Schochet \cite{Schochet1994limit}, Secchi \cite{Secchi2000}
			\end{tabular} & \begin{tabular}{c} 
			M\'etivier-Schochet \cite{Metivier2001limit, Metivier2003limit} \\
			Alazard \cite{Alazard2005limit}
			\end{tabular} \\
			\hline
			\end{tabular}
			~\\
			~\\
			Table 1: The existing literatures for the low Mach number limit of Euler equations.
		\end{center}	
		
		The study of singular limits in MHD are much less developed than that of Euler equations due to the strong coupling among the fluids, sound waves and magnetic fields, and there are still many unsolved problems since Majda raised open problems in this area in \cite[pp. 71-72]{Majdalimit}. For ideal MHD in the whole space, the study of low Mach number limit was established by Jiang, the first author and Li \cite{JJLMHDlimit2} and we also refer to Cheng, the first author and Schochet \cite{CJSAlfvenlimit1, CJSAlfvenlimit2} for the multi-scale singular limits in low Mach number and Alfv\'en number regime. 
		
		However, the singular limits of ideal MHD system in a domain with boundaries become much more subtle. Under the perfectly conducting wall condition ($B\cdot N=0$ on the boundary), Ohno-Shirota \cite{OS1998MHDill} proved that the linearized problem near a non-zero magnetic field is ill-posed in standard Sobolev spaces. The well-posedness results are referred to Yanagisawa-Matsumura \cite{1991MHDfirst} and Secchi \cite{Secchi1995, Secchi1995-2, Secchi1996, Secchi2002} under the setting of anisotropic Sobolev spaces, which were first introduced by Chen \cite{ChenSX}. In contrast, the corresponding incompressible problem is still well-posed in standard Sobolev spaces, see for example \cite{GuWang2016LWP}. The low Mach number limit for this problem was first proved by the first and the second authors \cite{JuWangMHDlimit} in the case of isentropic general data and non-isentropic well-prepared data. In \cite{JuWangMHDlimit}, a suitable closed subspace of standard Sobolev space, introduced by Secchi \cite{Secchi1995-2}, is used, but more restrictive constraints for the boundary value of initial data are required
		\[
			\p_3^{2k}(u_3,B_3)|_{\Sigma}=0,~\p_3^{2k+1}(p,u_1,u_2,B_1,B_2,S)|_{\Sigma}=0,~~k=0,1,\cdots,\lfloor\frac{m-1}{2}\rfloor,
		\]whose physical interpretation is still unclear. We also refer to a very recent result by Secchi \cite{Secchi2024} in which he proved the incompressible limit for the isentropic problem with general initial data by using another anisotropic Sobolev spaces defined in \cite{Secchi2002}. By observing a key structure, the second and the third authors \cite{WZ2023CMHDlimit} prove the incompressible limit for the non-isentropic problem with well-prepared initial data, and the function spaces that we used are exactly the same as the ones for their well-posedness, that is, the anisotropic Sobolev norms converge to an energy functional defined in standard Sobolev spaces. The framework of \cite{WZ2023CMHDlimit} is also generalized to the study of free-boundary problems \cite{Zhang2023CMHD1, Zhang2023CMHD2} by the third author. But {\it none of the existing works applies to the case of non-isentropic problems with general initial data.}
		
		We also remark that the direction of the magnetic field is crucially important for the study of ideal MHD in a domain with boundaries even if one only studies the local existence. For example, when the magnetic field is not tangential to the boundary, one can use the transversality of the magnetic field to compensate the loss of normal derivative and we refer to \cite{1987MHDfirst} for the local existence in standard Sobolev spaces. See also the results about the singular limits in \cite{JuMHD2} by the first author, Schochet, Xu and \cite{JuMHD1} by Jiang, the first author and Xu.
		
		The aim of this paper is to rigorously justify the singular limit  in low Mach number regime for ideal MHD with the perfectly conducting wall condition {\it in the non-isentropic case with general initial data}. The framework in this paper can cover the existing works and does not require extra boundary constraints as in \cite{JuWangMHDlimit}. The proof is based on the combination of several key observations and techniques: a special structure in vorticity analysis that illustrates the ``mismatch" between the anisotropic norms and the standard Sobolev norms, enhanced regularity of the entropy in the direction of the magnetic field, the usage of material derivative $D_t$ instead of $\p_t$ when defining the energy functional and the application of ``modified Alinhac good unknowns" to overcome the difficulty brought by the anisotropy of the function spaces.

		\subsection{Anisotropic Sobolev spaces}	
			Before stating our results, we should first define the anisotropic Sobolev space $H_*^m(\Omega)$ for $m\in\N$ and $\Om=\mathbb{R}^d_+=\{x\in\mathbb{R}^d: x_d>0 \}$, which was first introduced by Shuxing Chen \cite{ChenSX}. Let $\omega=\omega(x_d)$ be a cutoff function\footnote{The choice of $\omega(x_d)$ is not unique. We just need $\omega(x_d)$ vanishes on $\Sigma$,  comparable to the distance function near $\Sigma$ and comparable to 1 far away from $\Sigma$.} on $[0,+\infty)$ defined by $\omega(x_d)=\frac{x_d}{1+x_d}$. Then we define $H_*^m(\Omega)$ for $m\in\N^*$ as follows
			\[
			H_*^m(\Omega):=\left\{f\in L^2(\Omega)\bigg| (\omega\p_d)^{\alpha_{d+1}}\p_1^{\alpha_1}\cdots\p_d^{\alpha_d} f\in L^2(\Omega),~~\forall \alpha \text{ with } \sum_{j=1}^{d-1}\alpha_j +2\alpha_d+\alpha_{d+1}\leq m\right\},
			\]equipped with the norm
			\begin{equation}\label{anisotropic1}
				\|f\|_{H_*^m(\Omega)}^2:=\sum_{\sum_{j=1}^{d-1}\alpha_j +2\alpha_d+\alpha_{d+1}\leq m}\|(\omega\p_d)^{\alpha_{d+1}}\p_1^{\alpha_1}\cdots\p_d^{\alpha_d} f\|_{L^2(\Omega)}^2.
			\end{equation} For any multi-index $\alpha:=(\alpha_0,\alpha_1,\cdots,\alpha_{d},\alpha_{d+1})\in\N^{d+2}$, we define
			\[
			\p_*^\alpha:=\p_t^{\alpha_0}(\omega\p_3)^{\alpha_{d+1}}\p_1^{\alpha_1}\cdots\p_d^{\alpha_d},~~\lee \alpha\ree:=\sum_{j=0}^{d-1}\alpha_j +2\alpha_d+\alpha_{d+1},
			\]and define the \textbf{space-time anisotropic Sobolev norm} $\|\cdot\|_{m,*}$ to be
			\begin{equation}\label{anisotropic2}
				\|f\|_{m,*}^2:=\sum_{\lee\alpha\ree\leq m}\|\p_*^\alpha f\|_{L^2(\Omega)}^2=\sum_{\alpha_0\leq m}\|\p_t^{\alpha_0}f\|_{H_*^{m-\alpha_0}(\Omega)}^2.
			\end{equation}
			
			We also denote the interior Sobolev norm to be $\|f\|_{s}:= \|f(t,\cdot)\|_{H^s(\Omega)}$ for any function $f(t,x)\text{ on }[0,T]\times\Omega$ and denote the boundary Sobolev norm to be $|f|_{s}:= |f(t,\cdot)|_{H^s(\Sigma)}$ for any function $f(t,x)\text{ on }[0,T]\times\Sigma$. 
			
			From now on, we assume the dimension to be $d=3$, that is, $\Om=\mathbb{R}^3_+=\{x\in\mathbb{R}^3:x_3>0\}$ and $\Sigma=\{x_3=0\}$. The 2D case follow in the same manner as the 3D case up to a slight modification in vorticity analysis and we refer the details to \cite[Section 3.5]{WZ2023CMHDlimit} and no longer discuss it in this paper. 
		
		\subsection{Main results}	
		Now, we establish a local-in-time estimate that is uniform in Mach number $\eps$ for general initial data. 
		\begin{thm}[\textbf{Uniform-in-$\eps$ estimate}]\label{main thm, ill data}
			Let $\eps>0$ be fixed. Let $(q_0, u_0, B_0,S_0)\in H^8(\Om)\times H^8(\Om) \times H^8(\Om)\times H^9(\Om)$ be the initial data of \eqref{CMHD3} satisfying the compatibility conditions \eqref{comp cond} up to 7-th order and
			\begin{equation}
				E(0)\le M
			\end{equation}
			for some $M>0$ independent of $\eps$. Then there exists $T>0$ depending only on $M$, such that the solution to \eqref{CMHD3} satisfies the energy estimate
			\begin{equation}
				\sup_{t\in[0,T]}E(t) \le P(E(0)),
			\end{equation}
			where $P(\cdots)$ is a generic polynomial in its arguments, and the energy $E(t)$ is defined to be
			\begin{equation}
				\begin{aligned}\label{energy intro}
					E(t)=&~\sum_{l=0}^{4}E_{4+l}(t),\\
					E_{4+l}(t)=&\sum_{k=0}^{4-l}\left\| (\eps D_t)^{k+2l} (q,u,B,S,\rho^{-1}B\cdot\nab S)\right\|_{4-k-l}^2.
				\end{aligned}
			\end{equation}
			
		\end{thm}

			\begin{rmk}[Relations with anisotropic Sobolev space]
				The energy functional $E(t)$ above is considered as a variant of $\|\cdot\|_{8,*}$ norm at time $t>0$. In fact, the slip boundary condition implies that $D_t$ must be a tangential derivative on the boundary $\Sigma$, and so the ``anisotropic order" of $\|D_t^{k+2l} f\|_{4-k-l}$ is at most $(k+2l)+2\times(4-k-l)=8-k\leq 8$. The Mach number weights are determined according to the number of material derivatives such that the energy estimate for the above ``modified $\|\cdot\|_{8,*}$ norm" is uniform in $\eps>0$. In particular, the first-order time derivatives of $u, q$ are not bounded uniformly in $\eps$, which corresponds to the case of general initial data.
			\end{rmk}	

			\begin{rmk}[Well-posedness and regularity]
				For the initial data given in Theorem \ref{main thm, ill data}, the local well-posedness of \eqref{CMHD3} in $H_*^8(\Om)$ for a fixed $\eps>0$ has been proven in \cite[Theorem 2.1']{Secchi1996} or \cite[Theorem 1.1]{WZ2023CMHDlimit}, where the latter one also gives the uniform-in-$\eps$ estimates for well-prepared initial data. We assume the initial data belong to $H^8(\Om)$ instead of $H_*^8(\Om)$ only because we must guarantee the initial energy $E(0)<\infty$. There is no loss of regularity in the energy estimates. We choose $H_*^8(\Om)\hookrightarrow H^4(\Om)$ (instead of $H^3(\Om)$ as in Euler equations or elastodynamics \cite{WZ2024elasto}) because there are lots of terms in vorticity analysis of compressible ideal MHD that requires the $L^\infty$ boundedness of the second-order derivatives.
			\end{rmk}
			
			\begin{rmk}[Enhanced regularity of the entropy]
			We assume $S_0\in H^9(\Om)$ (actually $H_*^9(\Om)$) because we need the $H_*^8(\Om)$-regularity of the directional derivative $\rbp S$ to control the vorticity instead of the full $H_*^9(\Om)$ regularity. In fact, this enhanced ``directional" regularity of $S$ can be propagated from its initial data as we observe that $D_t$ commutes with $\rbp$. We refer to Section \ref{sect entropy} for the proof. 
			\end{rmk}
			
		The next main theorem concerns the low Mach number limit. We consider the inhomogeneous MHD equations together with a transport equation satisfied by $(u^0,B^0,\pi,S^0)$:
		\begin{equation} \label{IMHD}
			\begin{cases}
				\varrho(\p_t u^0 + u^0\cdot\nab u^0) -B^0\cdot\nab B^0+ \nab (\pi+\frac12|B^0|^2) =0&~~~ \text{in}~[0,T]\times \Omega,\\
				\p_t B^0+ u^0\cdot\nab B^0 -B^0\cdot\nab u^0=0 &~~~ \text{in}~[0,T]\times \Omega,\\
				\p_tS^0+u^0\cdot\nab S^0=0&~~~ \text{in}~[0,T]\times \Omega,\\
				\nab\cdot u^0=\nab\cdot B^0=0&~~~ \text{in}~[0,T]\times \Omega,\\
				u_3^0=B_3^0=0&~~~\text{on}~[0,T]\times\Sigma.
			\end{cases}
		\end{equation}

		\begin{thm}[\textbf{The low Mach number limit}] \label{main thm, limit}
			Under the hypothesis of Theorem \ref{main thm, ill data}, we assume that $(u_0,B_0,S_0) \to (u_0^0,B_0^0,S_0^0)$ in $H^4(\Omega)$ as $\eps\to 0$ with $\nab\cdot B_0^0=0$ in $\Om$ and $u_{0d}^0=B_{0d}^0=0$ on $\Sigma$, and that there exist positive constants $N_0$ and $\sigma$ such that $S_0$ satisfies
			\begin{equation}\label{EntropyDecay}
				|S_0(x)| \le N_0 |x|^{-1-\sigma},\quad |\nabla S_0(x)|\le N_0|x|^{-2-\sigma}.
			\end{equation}
			Then it holds that 
			\begin{align*}
				(q,u,B,S) \to& (0,u^0,B^0,S^0) \quad \text{weakly-* in } L^{\infty}([0,T];H^4(\Om)) \text{ and strongly in } L^2([0,T];H^{4-\delta}_{\mathrm{loc}}(\Om))
			\end{align*}
			for $\delta>0$. $(u^0,B^0,S^0)\in C([0,T];H^4(\Om))$ solves \eqref{IMHD} with initial data $(u^0,B^0,S^0)|_{t=0}=(w_0,B^0_0,S_0^0)$, that is, the incompressible MHD system together with a transport equation of $S^0$, where $w_0\in H^4(\Omega)$ is determined by
			\begin{equation}
				w_{0d}|_{\Sigma}=0,\quad \nabla\cdot w_0=0,\quad \nabla\times (\rho(0,S^0_0)w_0 )=\nabla\times(\rho(0,S_0^0)u_0^0).
			\end{equation}
			Here $\varrho$ satisfies the transport equation $$\p_t\varrho+u^0\cdot\nab \varrho =0,~~\varrho|_{t=0}=\rho(0,S_0^0).$$ 
			The function $\pi$ satisfying $\nab\pi\in C([0,T];H^3(\Om))$ represents the fluid pressure for incompressible MHD system \eqref{IMHD}. In the case of $d=2$, $\nabla\times (\rho(0,S^0_0)w_0 )=\nabla\times(\rho(0,S_0^0)u_0^0)$ should be replaced by $\nabla^\perp\cdot (\rho(0,S^0_0)w_0 )=\nabla^\perp\cdot(\rho(0,S_0^0)u_0^0)$ with $\nab^\perp=(-\p_2, \p_1)^\mathrm{T}.$
		\end{thm}
		
		\begin{rmk}[Boundedness issue of the domain]
			In the proof of the Theorem \ref{main thm, ill data}, the uniform estimates can be established regardless of the boundedness of the domain. However, in the proof of the Theorem \ref{main thm, limit}, the unboundedness of the domain, for example assuming $\Omega$ to be the half-space or the exterior of a compact smooth domain in $\mathbb{R}^d$, is necessary due to the (global) dispersion property for the wave equation. The strong convergence in time can be obtained by proving local energy decay via the defect measure technique \cite{Metivier2001limit, Alazard2005limit}. The entropy decay condition \eqref{EntropyDecay} is also needed due to H\"ormander \cite[Theorem 17.2.8]{Hormander3} used in the proof of strong convergence.
		\end{rmk}		
				
		\subsection{Organization of the paper}\label{sect org} 
		This paper is organized as follows. In Section \ref{sect overview intro}, we discuss the main difficulties and briefly introduce our strategies to tackle the problem. Then Section \ref{sect uniform} is devoted to the proof of uniform estimates in Mach number. The strong convergence to the inhomogeneous MHD system is established in Section \ref{sect limit} by using the technique of microlocal defect measure as in \cite{Metivier2001limit, Alazard2005limit}.

		\subsection*{List of Notations}
		\begin{itemize}
			\item ($L^\infty$-norm) $\|\cdot\|_{\infty}:= \|\cdot\|_{L^\infty(\Omega)}$, {$|\cdot|_{\infty}:=\|\cdot\|_{L^\infty(\Sigma)}$}. 
			\item (Interior Sobolev norm) $\|\cdot\|_{s}$:  We denote $\|f\|_{s}:= \|f(t,\cdot)\|_{H^s(\Omega)}$ for any function $f(t,y)\text{ on }[0,T]\times\Omega$.
			\item  (Boundary Sobolev norm) $|\cdot|_{s}$:  We denote $|f|_{s}:= |f(t,\cdot)|_{H^s(\Sigma)}$ for any function $f(t,y)\text{ on }[0,T]\times\Sigma$.
			\item (Polynomials) $P(\cdots)$ denotes a generic polynomial in its arguments with non-negative coefficients.
			\item (Commutators) $[T,f]g=T(fg)-f(Tg)$, $[f,T]g=-[T,f]g$ where $T$ is a differential operator and $f,g$ are functions.
			\item (Equality modulo lower order terms) $A\lleq B$ means $A=B$ modulo lower order terms. 
		\end{itemize}

		\section{Difficulties and strategies}\label{sect overview intro}	
			Before going to the detailed proofs, we briefly discuss the main difficulties in this problem and our strategies to overcome these difficulties. In particular, we focus on the difficulties that do not appear in neither of Euler equations, isentropic MHD equations or the case of well-prepared initial data.
			
			The system of compressible ideal MHD with the perfectly conducting wall conditions is a first-order hyperbolic system with characteristic boundary conditions of constant multiplicity, for which there is a potential of normal derivative loss. This also happens for Euler equations and elastodynamics inside a rigid wall with the slip boundary condition, but the vorticity for those two equations can be controlled in standard Sobolev spaces (see Schochet \cite{Schochet1987limit} and the third author \cite{Zhang2021elasto}), which could compensate the loss of normal derivative. However, for compressible ideal MHD, one encounters a loss of normal derivative in vorticity analysis caused by the simultaneous appearance of compressibility and the magnetic field. That is exactly one has to use the anisotropic Sobolev spaces to prove the local existence.
		
		\subsection{Motivation to define $E(t)$: a special structure in vorticity analysis}\label{sect stat 1}
		We are inspired by a special structure of Lorentz force in vorticity analysis to define the energy functional $E(t)$ as in \eqref{energy intro}. Let us take the estimate of $\|\nab\times u\|_3^2$ in the control of $\|u\|_4^2$ for example.
		\subsubsection{Substitution of the coefficient}
		 In the case of non-isentropic problem with general initial data, one has to replace the coefficient $\rho(\eps q,S)$ by $\rho_0:=\rho(0,S)$ to derive the evolution equation of the vorticity and the current, namely,
		\[
D_t(\rho_0 u)-\rho_0 (\rho^{-1} B\cdot\nab) B = -\underbrace{\nab(\eps^{-1} q + \frac12 |B|^2)}_{\text{curl-free}} +\underbrace{\frac{\rho - \rho_0}{\rho}}_{=O(\eps)}\nab(\eps^{-1} q + \frac12 |B|^2),
		\] as stated in M\'etivier-Schochet \cite{Metivier2001limit}. Or else, there must be terms like $\p^\alpha \rho (\nab\times D_t u)$ whose coefficient $\p^\alpha\rho = O(\eps)\p^\alpha q + O(1)\p^\alpha S$ is $O(1)$ size and so this term is not uniformly bounded. The advantage of such substituiton is that the unbounded terms are avoided and the extra-generated term has small coefficient and its leading-order part is curl-free. Then it is easy to reproduce the control of $\nab\times(u,B)$ from that of $\nab\times(\rho_0 u, \rho_0 B)$ because $\rho_0$ only depends on $S$ whose estimate is straightforward.
		
		\subsubsection{A weighted anisotropic structure contributed by the Lorentz force}
		In the analysis of $\|\nab\times (\rho_0u)\|_3^2$, we must encounter a fifth-order term
		\[
		\mathcal{K}:=-\int_{\Om}\p^3 \nab\times (\rho_0 B) \cdot \p^3 \nab\times (\rho_0\rho^{-1} B\nab\cdot u) \dx,
		\]which is uncontrollable in the setting of $H^4$. But if we insert the continuity equation $\nab\cdot u=-\eps aD_t  q$, commute $\nab\times$ with $D_t$ and insert the momentum equation $-\nab q=\eps(\rho D_t u +B\times(\nab\times B))$, we can see that
		\begin{align*}
		\p^3 \nab\times (\rho_0\rho^{-1} B\nab\cdot u) = -\eps^2a\rho^{-1} B\times \left[B\times D_t(\p^3\nab\times(\rho_0 B))\right] + \eps^2\rho_0 a B\times (\p^3 D_t^2 u)+ \text{low-order terms}.
		\end{align*} On the right side, the first term contributes to an energy term $-\frac12\ddt \io a\rho^{-1}|\eps B\times(\p^3\nab\times(\rho_0 B))|^2\dx$ that gives the regularity of the Lorentz force, while the second term exhibits a ``weighted" anisotropic structure that indicates us to trade one normal derivative (in the curl operator $\nab\times$) for the second-order tangential derivative $\eps^2D_t^2$. As $\eps\to 0$, this anisotropic part converges to 0, and so the incompressible problem can be studied in standard Sobolev spaces. For details, we refer to Section \ref{sect curl}.
		
		\subsubsection{The whole reduction procedures}
		Apart from the vorticity estimates, we also need to consider the divergence estimates. The continuity equation indicates us to trade one normal derivative in $\nab\cdot u$ for one tangential derivative $\eps D_t q$. Then for the pressure $q$, we can always invoke the momentum equation $-\nab q=\eps(\rho D_t u + B\times(\nab\times B))$ to convert a spatial derivative to $\eps D_t u$ and the estimate of current $\nab\times B$ until there is no spatial derivative falling on $q$ and no spatial derivative falling on $u$. The full details are referred to Section \ref{sect reduction q}.
		
		It should be remarked that all tangential derivatives generated in the above reduction procedures are material derivatives, and that explains why we write $D_t$ instead of $\p_t$ when defining $E(t)$. In Section \ref{sect stat 3}, we will see using $D_t$ instead of $\p_t$ is necessary for the non-isentropic problem with general initial data, as opposed to the case of well-prepared data or the isentropic problem in \cite{JuWangMHDlimit, WZ2023CMHDlimit, Secchi2024}. 
		
		We conclude the reduction scheme in the following two diagrams
		{\small\[
\begin{tikzcd}
\|(u,B)\|_4 \arrow[d, "~~", "\text{div}" ] \arrow[r, "~~", "{\text{curl}}" ]&\|(\eps D_t)^2(u,B)\|_3 \arrow[d, "~~", "\text{div}" ]\arrow[r, "~~", "{\text{curl}}" ]&\|(\eps D_t)^4(u,B)\|_2 \arrow[d, "~~", "\text{div}" ]\arrow[r, "~~", "{\text{curl}}" ]&\|(\eps D_t)^6(u,B)\|_1 \arrow[d, "~~", "\text{div}" ]\arrow[r, "~~", "{\text{curl}}" ]&\|(\eps D_t)^8(u,B)\|_0\arrow[dd, "~~", "~~" ] \\
\|\eps D_t q\|_3\arrow[d, "~~", "\text{div reduction}" ]\arrow[ur,"~~","~~"] &\|(\eps D_t)^3 q\|_2\arrow[d, "~~", "\text{div reduction}" ]\arrow[ur,"~~","~~"]&\|(\eps D_t)^5 q\|_1\arrow[d, "~~", "\text{div reduction}" ] \arrow[ur,"~~","~~"]&\|(\eps D_t)^7 q\|_0\arrow[d, "~~", "\text{div reduction}" ]\\
{(\eps D_t)^4\text{-estimate}}\arrow[rrd, "~~", "~~" ]&{(\eps D_t)^5\text{-estimate}}\arrow[rd, "~~", "~~" ]&{(\eps D_t)^6\text{-estimate}}\arrow[d, "~~", "~~" ]&{(\eps D_t)^7\text{-estimate}}\arrow[ld, "~~", "~~" ]&{(\eps D_t)^8\text{-estimate}}\arrow[lld, "~~", "~~" ]\\
&&\text{closed}&&
\end{tikzcd}
		\]}
		\begin{center}
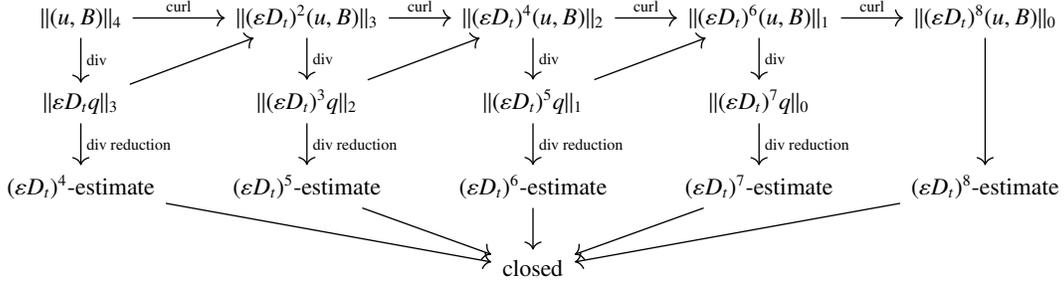
		Diagram 1: the reduction scheme for the curl part.
		\end{center}
		In the above diagram, the ``div reduction" refers to the reduction procedure for divergence and pressure in the previous paragraph. It can also be presented as the diagram below (we take $l=0$ for example).
		{\small\[
		\begin{tikzcd}
		&&\text{Further div-curl}&\\
		&&E_5(t)\arrow[u, "~~", "~~" ]&\\
		\|(u,B)\|_4\arrow[urr, "~~", "\text{curl}" ]  \arrow[dr, "~~", "{\text{div}}" ]&\|\eps D_t(u,B)\|_3\arrow[ur, "~~", "\text{curl}" ] \arrow[dr, "~~", "{\text{div}}" ]&\|(\eps D_t)^2(u,B)\|_2  \arrow[dr, "~~", "{\text{div}}" ]\arrow[u, "~~", "\text{curl}" ]&\|(\eps D_t)^3(u,B)\|_1 \arrow[ul, "~~", "\text{curl}" ] \arrow[dr, "~~", "{\text{div}}" ]&\|(\eps D_t)^4(u,B)\|_0\arrow[d, "~~", "\text{together}" ] \\
		\|q\|_4\arrow[ur, "~~","\text{eq}" ] \arrow[u, "~~","\text{eq}" ] & \|\eps D_t q\|_3 \arrow[ur, "~~","\text{eq}" ] \arrow[u, "~~","\text{eq}" ]& \|(\eps D_t)^2 q\|_2\arrow[ur, "~~","\text{eq}" ] \arrow[u, "~~","\text{eq}" ] & \|(\eps D_t)^3 q\|_1 \arrow[ur, "~~","\text{eq}" ] \arrow[u, "~~","\text{eq}" ] & \|(\eps D_t)^4 q\|_0\arrow[d, "~~","\text{tangential}" ] \\
		&&&&\text{closed}
		\end{tikzcd}
		\]}
		\begin{center}
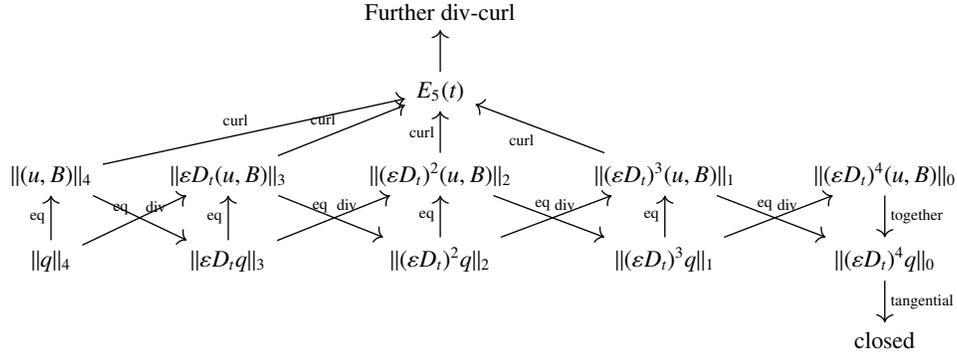
		Diagram 2: the reduction scheme for the divergence part in $E_4(t)$.
		\end{center}
		Here ``eq" means invoking the momentum equation and ``div" means invoking the continuity equation.
		
		Now, we can explain why we define $E(t)$ to be 
		\[
		\sum_{l=0}^4\sum_{k=0}^{4-l}\left\| (\eps D_t)^{k+2l} (q,u,B,S,\rho^{-1}B\cdot\nab S)\right\|_{4-k-l}^2
		\] Let $l\in\{0,1,2,3\}$ be the number of times that we do the vorticity estimates. After such $l$ times of reduction for $\|u, B\|_4$, we have at most $4-l$ normal derivatives left, but we also obtain $(\eps D_t)^{2l}$ falling on each variable. This step corresponds to the first row of Diagram 1 and the $(\eps D_t)^{2l}$-part in the energy functional. Then for each fixed $l\in\{0,1,2,3\}$, Diagram 2 is parallel to the study of Euler equations, that is, one trades one spatial derivative for one $\eps D_t$ in the divergence part. Finally, we only need to control the estimates of full material derivatives which are purely tangential derivatives and this step is analyzed in Section \ref{sect tg}.
		
		\subsection{Enhanced ``directional" regularity of the entropy}\label{sect stat 2}
		One may notice that we require the directional derivative $\rbp S$ has the same regularity as the other variables. This is in general incorrect for a first-order hyperbolic system, but it can really be achieved for compressible ideal MHD as long as we assume the initial data of $\rbp S$ has the corresponding regularity. This fact is easy to prove, as \textit{we observe that $D_t$ commutes with $\rbp$} which then leads to $D_t(\rbp S)=0$ and the desired enhanced regularity. In fact, this is even easier to be observed if we study the free-boundary MHD in Lagrangian coordinates, e.g., \cite{Zhang2021CMHD}, in which the material derivative becomes $\p_t$ and $\rbp$ becomes {\it time-independent}!
		
		Such enhanced directional regularity is necessary to control the vorticity and the current. For example, in the analysis of $\|\nab\times(\rho_0 u)\|_3$, we must encounter the following two terms
		\begin{align*}
		-\io\p^3\nab\times\left( \underline{(\rho^{-1} B\cdot\nab \rho_0)}B\right)\cdot\p^3\nab\times(\rho_0 u)\dx- \io \p^3\nab\times\left(\underline{(\rho^{-1} B\cdot\nab \rho_0)}u\right)\cdot\p^3\nab\times(\rho_0 B)\dx,
		\end{align*}in which the underlined terms essentially require the $H^4$ regularity of $\rbp S$. Such difficulty never appears in the study of Euler equations, isentropic problems, or non-isentropic problems with well-prepared initial data. On the other hand, our observation exactly resolves this issue. Similar argument also applies to the recent work by the second author and the third author \cite{WZ2024elasto} about neo-Hookean elastodynamics.

		\subsection{Usage of material derivatives and ``modified" Alinhac good unknowns}\label{sect stat 3}
		Now, we turn to explain why we shall use the variable-coefficient derivative $D_t$ instead of the simpler one $\p_t$ to define $E(t)$. In fact, this is due to the unboundedness of the first-order time derivatives in the case of general initial data. If we decompose the material derivative $D_t$ into $\p_t$, tangential spatial derivative $\TP=\p_1$ or $\p_2$ and the co-normal part $\omega(x_3)\p_3$, then we must separately do the tangential estimates for $\TT=\p_t,\TP,\omega(x_3)\p_3$ as in the case of well-prepared data  \cite{WZ2023CMHDlimit}. However, this will again produce singular terms such as $((\eps\TP)^8\rho )(D_t u)$ and we cannot substitute $\rho$ by $\rho_0$ as in the vorticity control. To overcome this difficulty, we must ensure every tangential derivative to be the material derivative $D_t$ such that we could completely eliminate the $O(1)$-size terms $\TT^\alpha \rho \approx O(\eps)\TT^\alpha q + O(1)\TT^\alpha S$ by using $D_t S=0$.
		
		Under our setting, a new difficulty appears, that is, $D_t$ does not commute with the $\nab$. In Lemma \ref{Prop_AGU}, we compute the concrete form of $[\p, D_t^k]f$ and find that each term in this commutator must be a monomial of $(\p D_t^j f)$ and $(\p D_t^m u)$ for some $j, m\in\Z$. What's more difficult is that the top-order commutator, namely $[\p, D_t^8]f$, contains $(\p D_t^7u)(\p f)$ and $8(\p u)(\p D_t^7 f)$ whose $L^2(\Om)$ norms cannot be directly controlled when $\p=\p_3$. To overcome this difficulty, we adopt the idea of Alinhac good unknowns \cite{Alinhac1989good}, which reveals that the essential leading-order part in $D_t^8\p_i f$ is not simply $\p_i D_t^8 f$ but actually the $\p_i$ derivative of the Alinhac good unknowns (denoted by $\mathbf{F}^8$)
		\[
		D_t^8 \p_i f \neq \p_i D_t^8 f + L^2(\Om)\text{-controllable terms}\xRightarrow{~~~~~~} D_t^8 \p_i f = \p_i \mathbf{F}^8 + L^2(\Om)\text{-controllable terms}.
		\] In other words, the core idea of using Alinhac good unknowns is to rewrite the uncontrollable terms to be the form $\p_i(L^2\text{-controllable})$ terms and merge the terms in the parenthesis into the leading-order term, such that the covariance of the system is still preserved after being differentiated by variable-coefficient derivatives. In view of the singular limit problems, this step is used to preserve the anti-symmetric structure $E_0(U)\p_t U + \eps^{-1}\nab U = f$ for the differentiated system.
		
		Under the setting of standard Sobolev spaces, $\mathbf{F}^8$ is defined to be $D_t^8 f-D_t^7 u\cdot\nab f$, see for example \cite{GuWang2016LWP}. Under the setting of anisotropic Sobolev spaces, the third author \cite{Zhang2021CMHD} observed that the merged terms must be modified according as the variable $f$. After careful and delicate calculations, the ``modified" Alinhac good unknowns used in this paper are (see also \cite[Section 5.5]{Zhang2021CMHD})
		\[
		\mathbf{U}_i^8:=D_t^8 u_i - 9 D_t^7 u\cdot \nab u_i,\q\q \mathbf{Q}^8:= D_t^8\qc - D_t^7u\cdot\nab \qc,~~\qc:=q+\frac{\eps}{2}|B|^2.
		\]Such difference is due to the fact that we must preserve the divergence structure for the velocity but there is no such structure for the pressure. We refer the details to Section \ref{sect tg}.

		\section{Uniform energy estimates}\label{sect uniform}		
		This section is devoted to the proof of Theorem \ref{main thm, ill data}, which will be split into several parts: $L^2$ estimates, entropy estimates, tangential estimates for material derivatives and the reduction of normal derivatives via div-curl analysis.
		\subsection{$L^2$ estimate}\label{sect L2}
		First, we establish $L^2$-energy estimate for \eqref{CMHD3}.
		\begin{prop}\label{lem L2} We define the $L^2$ energy $E_0(t)$ by
		\[
		E_0(t):=\frac{1}{2} \int_{\Om} a q^2 + \rho |u|^2 + |B|^2 +S^2 \dx.
		\] Then it satisfies
		\begin{align}
		\frac{\mathrm{d}E_0(t)}{\dt}\leq P(E_4(t)).
		\end{align}
		\end{prop} Invoking the momentum equation and integrate by parts, we have:
		\begin{align}
			&\frac{1}{2}\ddt \int_{\Om} a q^2 + \rho |u|^2 + |B|^2 +S^2 \dx\notag\\
			=&\frac{1}{2}\int_{\Om} (\p_t a +\nabla\cdot(a u))q^2+ \nab\cdot u( |B|^2+S^2) \dx- \eps^{-1}\int_{\Om}\nab\cdot(q u) \dx\notag\\
			&+\int_{\Om}(B\cdot\nab B)\cdot u + (B\cdot\nab u)\cdot B \dx -\int_{\Om} \frac{1}{2}\nab|B|^2 \cdot u + (B\cdot B)\nab\cdot u \dx	\notag\\
			=&\frac{1}{2}\int_{\Om} (\eps\p_q a D_t q + a \nabla \cdot u)q^2+ \nab\cdot u( |B|^2+S^2) \dx,
		\end{align}
		which give the desired energy estimate. This also implies that	
		\begin{align}\label{L2}
			 \|(q, u, B, S)\|_0^2\lesssim E_4(0) + \int_0^t P(E_4(\tau))\dtau.
		\end{align}
		\subsection{Enhanced regularity of the entropy}\label{sect entropy}
		
			We now verify that the directional derivative $\rho^{-1}B\cdot\nab$ commutes with the material derivative $D_t$. This property leads to the enhanced regularity of the entropy in the direction of the magnetic field.
		\begin{prop}\label{Prop_Com}
			Let $(q,u,B,S)$ be a solution to system \eqref{CMHD3}, then the following equality holds:
			\begin{equation}
				[\rho^{-1} B\cdot\nab, D_t]=0.
			\end{equation}
		\end{prop}
		\begin{proof}
			By using equations of $\rho$ and $B$, we have that
			\begin{align}
				D_t (\rho^{-1} B_i )=&~ \rho^{-1} D_t B_i + B_i D_t \rho^{-1}\notag\\
				=&~\rho^{-1}(B\cdot \nab u_i -B_i\nab\cdot u) -\rho^{-2}B_i D_t \rho \notag\\
				=&~\rho^{-1}( B\cdot \nab u_i -B_i\nab\cdot u) +\rho^{-1}B_i\nab\cdot u\notag\\
				=&~\rho^{-1}B\cdot \nab u_i,\quad i=1,\cdots,d.
			\end{align}
			For any function $f\in C([0,T];H^2)\cap C^1([0,T];H^1)$, we have that
			\begin{align}
				D_t( \rho^{-1} B\cdot\nab f)=&~  \sum_{k,j=1}^{d}\left( \p_t (\rho^{-1} B_k\p_k f) +  u_j\p_j( \rho^{-1} B_k\p_k f )\right)\notag\\
				=&~ \sum_{k,j=1}^{d} \left( \rho^{-1}B_k\p_k (\p_t f +u_j\p_j f) + \p_t (\rho^{-1}B_k) \p_k f -\rho^{-1} B_k\p_k u_j\p_j f +u_j\p_j(\rho^{-1}B_k)\p_k f \right) \notag\\
				=&~(\rho^{-1}B\cdot\nab) D_t f + \sum_{k=1}^{d} (  D_t(\rho^{-1}B_k )-\rho^{-1} B\cdot\nab u_k )\p_k f=(\rho^{-1}B\cdot\nab) D_t f.
			\end{align}
			
		\end{proof}

		Since $D_t S=0$ and $D_t(\rho^{-1}B\cdot\nab S)=0$, we can easily prove the estimates for $(S,\rho^{-1}B\cdot\nab S)$. The proof does not involve any boundary term because we do not integrate by parts, so we omit the details.
		\begin{cor}\label{Cor_Entropy} The entropy $S$ in \eqref{CMHD3} satisfies the following estimates:
		\begin{equation}
			\ddt \sum_{l=0}^4\sum_{k=0}^{4-l}\left\|(\eps D_t)^{k+2l}(S,\rho^{-1}B\cdot\nab S)\right\|_{4-k-l}^2\leq P(E(t)).
		\end{equation}
		\end{cor}

		\subsection{Tangential energy estimates}\label{sect tg}
		
		In this section, we will derive the energy estimates for tangential derivatives of solutions. It is more convenient to use the following equations derived from \eqref{CMHD3}:
		\begin{equation}\label{CMHD4}
			\begin{cases}
				a D_t \qc -\eps a B \cdot D_t B  + \eps^{-1} \nab\cdot u=0~~~&\text{in}~[0,T]\times\Om, \\
				\rho D_t u+\eps^{-1}\nab \qc -B\cdot\nabla B=0~~~& \text{in}~[0,T]\times\Om, \\
				D_t B -B\cdot \nab u -\eps a B( D_t \qc -\eps B\cdot D_t B)=0~~~&\text{in}~[0,T]\times\Om, \\
				\nabla\cdot B=0~~~&\text{in}~[0,T]\times\Om,\\
				D_t S=0~~~&\text{in}~[0,T]\times\Om,\\
				a=a(\eps \qc-\eps^2|B|^2/2,S)>0,\quad \rho=\rho(\eps \qc-\eps^2|B|^2/2,S)>0~~~&\text{in}~[0,T]\times\bar{\Om},
			\end{cases} 
		\end{equation}
		where $\qc = q+ \eps|B|^2/2$ is the total pressure.

		\begin{prop}\label{Lem_TangentialEstimates}
			It holds that
			\begin{align}
				\sum_{k=1}^{8} \sup_{\tau\in[0,t]}\left\|(\eps D_t)^k(q,u,B)\right\|_0^2 \le	P(E(0)) +  E(t)\int_0^t P(E(\tau)) \mathrm{d}\tau.	
			\end{align}
			
		\end{prop}
		
		The proof of this proposition is divided into two parts: When $1\leq k\leq 7$, one can prove the tangential estimates by straightforward calculations; when $k=8$, the anisotropy of the function spaces leads to potential loss of regularity in the commutator $[\nab, D_t^8]$ as the terms having the form $\nab D_t^7 f~(f=u,q)$ must appear.
		\subsubsection{Low-order estimates for $\leq 7$ material derivatives}
		
			For $1\le k\le  7$,  we take $D_t^k$ to the equations \eqref{CMHD4} to obtain that
			\begin{equation}\label{Equ_AGU}
				\begin{aligned}
					&a D_t^{k+1}\qc -\eps a B\cdot D_t^{k+1} B +\eps^{-1}\nab\cdot D_t^k u= \RR_q^k,\\
					&\rho D_t^{k+1} u + \eps^{-1}\nab D_t^k q -B\cdot\nab D_t^k B =\RR_u^k,\\
					& D_t^{k+1}B  - B\cdot\nab D_t^k u -\eps a B ( D_t^{k+1}\qc -\eps B\cdot D_t^{k+1}B )=\RR_B^k,
				\end{aligned}
			\end{equation}
			where
			\begin{equation*}
				\begin{aligned}
					\RR_q^k=&~ [a , D_t^k] D_t \qc-\eps[a B , D_t^k] \cdot D_t B+ \eps^{-1}[\p_j, D_t^k]u_j,\\
					\RR_u^k=&~ [\rho , D_t^k] D_t u +\eps^{-1}[\nab, D_t^k]q -[B\cdot\nab,D_t^k] B ,\\
					\RR_B^k=&~ -[B\cdot\nab, D_t^k]u - \eps [a B, D_t^k ] D_t\qc + \eps^2 [a B B_j, D_t^k ] D_t B_j   .
				\end{aligned}
			\end{equation*}
			Multiplying \eqref{Equ_AGU} by $\varepsilon^{2k}D_t^k(\qc,u, B)$, integrating over $\Omega$ and integrating by parts, one sees that
			\begin{align}
				&\frac{\varepsilon^{2k}}{2} \ddt\int_{\Omega}\left( a|D_t^k \qc-B\cdot D_t^k B|^2 +\rho|D_t^k u|^2  +|D_t^k B|^2\right) \dx \notag\\
				=&~\frac{\varepsilon^{2k}}{2}\int_{\Omega}\left( (\partial_t a+ \nabla\cdot(a u))|D_t^k \qc-B\cdot D_t^k B|^2 +\nabla\cdot u|D_t^k B|^2  \right)\dx\notag\\
				&~-\varepsilon^{2k}\int_{\Omega} a (\eps D_t B \cdot D_t^k B)( D_t^k\qc-B\cdot D_t^k B ) \dd x\notag\\
				&~+\varepsilon^{2k}\int_{\Omega} (   (B\cdot\nab) D_t^k B    )\cdot D_t^k u + ((B\cdot\nabla)D_t^k u)\cdot D_t^k B  \dx \notag\\
				&~+\varepsilon^{2k}\int_{\Omega} \left( D_t^k \qc \RR_q^k + D_t^k u\cdot \RR_u^k+ D_t^k B\cdot \RR_B^k \right)\dx\notag\\
				:=&~\mathcal{G}_1+\mathcal{G}_2+\mathcal{G}_3+\mathcal{G}_4,\label{energy estimate_Tan4}
			\end{align}
			where the equation $\partial_t \rho+ \nabla\cdot(\rho u)=0$ and the boundary condition $u_3=B_3=0$ on $\Sigma$ are used. We now have to control each term on the right-hand side of (\ref{energy estimate_Tan4}). Since $a$ is smooth funtions of $\varepsilon q$ and $S$, we employ the Sobolev inequality to deduce that
			\begin{align*}
				\|\partial_t a\|_{\infty}=\|\eps\p_q\rho \p_t q - \p_S \rho (u\cdot\nabla) S\|_{\infty}\le &~ P(E(t)),\\
				\| \nabla\cdot(a u)\|_{\infty}+\|\nabla\cdot u \|_{\infty}\le&~ P(E(t)).
			\end{align*}
			Thus, it is easy to find that $\mathcal{G}_1$ can be bounded from above by
			\begin{align}\label{Estimate_G1}
				|\mathcal{G}_1|\lesssim    \|\partial_t a +\nabla\cdot(a u)\|_{\infty}\| \eps^k(D_t^k\qc-B\cdot D_t^k B)\|_0^2 +\|\nabla\cdot u \|_{\infty} \|\eps^k D_t^k B\|_0^2 \le P(E(t)).
			\end{align}
			Similarly, we control $\mathcal{G}_2$ as follows:
			\begin{equation}
				|\mathcal{G}_2|\lesssim \|\eps D_t B\|_{\infty} \|\eps^k D_t^k B\|_0 \|\eps^k(D_t^k \qc- B\cdot D_t^k B)\|_0 \le P(E(t)).
			\end{equation}

			For $\mathcal{G}_3$, integrating by parts and using the boundary conditions $u_3=B_3=0$ on $\Sigma$ and $\nabla\cdot B=0$, one obtains that
			\begin{equation}
				\mathcal{G}_3=\eps^{2k}\int_{\Omega}  -((B\cdot\nabla)D_t^k u)\cdot  D_t^k B+ ((B\cdot\nabla)D_t^k u)\cdot  D_t^k B  \dx=0.
			\end{equation}

			Recall that Lemma \ref{Prop_AGU} indicates that the highest-order terms in the commutator $\eps^k[\nab, D_t^k]f$  have the form of either $(\eps^{k-1}\p D_t^{k-1}u)(\eps\p f)$ or $(\eps\p u)(\eps^{k-1}\p D_t^{k-1}f)$ whose $L^2(\Om)$ norms can be directly controlled by $P(E(t))$ thanks to $k\leq 7$. By standard product estimates, it is straightforward to see that the $L^2$-norm of terms in $\mathcal{G}_4$ can be controlled by $P(E(t))$, i.e.
			\begin{equation}\label{Estimate_G4}
				\mathcal{G}_4 \le P(E(t)).
			\end{equation}
		 Substituting the estimates for $\mathcal{G}_1$ - $\mathcal{G}_4$ into \eqref{energy estimate_Tan4}, we obtain that
			\begin{align}
				\ddt \left\|\eps^k D_t^k(\qc,u, B)\right\|_0^2 \le P(E(t)),\quad 1\le k\le 7.
			\end{align}
		By the definition of $\qc$, it is easy to verify that
		\begin{equation}
			\|\eps^k D_t^k q\|_0^2 \le P(E(0)) +\int_0^t P(E(\tau)) \dd \tau,\quad 1\le k\le 7.
		\end{equation}
			
			\subsubsection{Estimates of 8 material derivatives}
			
			Now we consider the estimates of the rest of highest order tangential derivatives $D_t^8$. In such case, the commutator $\eps^8[\nab,D_t^8]f$ contains the terms $(\eps^7\p D_t^7u)(\eps\p f)$ and $(\eps\p u)(\eps^7\p D_t^7f)$ in which the $\p D_t^7(\cdot)$ part is not controllable in $L^2(\Om)$. Such possible loss of regularity is completely caused by the anisotropy of the function spaces, as the $\p$ may be a normal derivative which is considered to be second-order. To overcome this difficulty, we shall introduce the ``modified Alinhac good unknowns" in \cite[Section 5.5]{Zhang2021CMHD} by the third author. To be more precise, let
			\begin{equation}
				\QQ^8= D_t^8 \qc- D_t^{7} u_i\p_i \qc,\quad \UU^8= D_t^8 u-9 D_t^{7} u_i\p_i u,
			\end{equation}
			respectively be the Alinhac good unknowns of $q$ and $u$. Therefore, by Lemma \ref{Prop_AGU} and Proposition \ref{Prop_Com}, we have that
			\begin{align}
				D_t^8 \nab \qc=&~  \nab \QQ^8 -8\nab u_i\p_i D_t^{7} \qc- Z^8 \qc,\\
				D_t^8 \nab \cdot u=& ~ \nab\cdot  D_t^8 u - \nab\cdot( D_t^7 u_i\p_i u ) -8 \p_j u_i\p_i D_t^7 u_j - Z_j^8 u_j\notag\\
				=&~ \nab\cdot\UU^8 +8\p_j (\nab\cdot u) D_t^{7}u_j- Z^8_j  u_j,\\
				D_t^8(B\cdot \nab B)=D_t^8(\rho\times\rho^{-1}B\cdot \nab B)=&~ [D_t^8,\rho](\rho^{-1}B\cdot\nab B) + (B\cdot\nab)(D_t^8 B),\\
				D_t^8(B\cdot \nab u)=D_t^8(\rho\times\rho^{-1}B\cdot \nab u)=&~ [D_t^8,\rho](\rho^{-1}B\cdot\nab u) + B\cdot\nab \UU^8 +9(B\cdot\nab)(D_t^7 u_i\p_i u),
			\end{align}where $Z^k=(Z_1^k,Z_2^k,Z_3^k)^\mathrm{T}$ represents the low-order terms in the commutator and its precise definition is recorded in Lemma \ref{Prop_AGU}.
			We take $D_t^8$ to the equations \eqref{CMHD4} to obtain that
			\begin{equation}\label{Equ_AGU8}
				\begin{aligned}
					&a D_t \QQ^8 -\eps a B\cdot D_t^9 B +\eps^{-1}\nab\cdot \UU^8= \RR_q^8,\\
					&\rho D_t \UU^8 + \eps^{-1}\nab \QQ^8 -(B\cdot\nab)(D_t^8 B)  =\eps^{-1} 8 \nab u_i\p_i D_t^7 \qc + \RR_u^8,\\
					& D_t^9 B  - B\cdot\nab \UU^8 -\eps a B ( D_t \QQ^8 -\eps B\cdot D_t^9 B )=-9(B\cdot\nab)(D_t^7 u_i\p_i u)+\RR_B^8,
				\end{aligned}
			\end{equation}
			where
			\begin{equation*}
				\begin{aligned}
					\RR_q^8=&~ [a , D_t^8] D_t \qc- a D_t( D_t^{7} u_i\p_i \qc )-\eps[a B , D_t^8] \cdot D_t B+ \eps^{-1}(   Z_j^8 u_j -8\p_j(
					\nab\cdot u) D_t^7 u_j),\\
					\RR_u^8=&~ [\rho , D_t^8] D_t u - 9\rho D_t( D_t^7 u_i\p_i u )+ \eps^{-1}Z^8 \qc - [D_t^8,\rho](\rho^{-1} B\cdot\nab B),\\
					\RR_B^8=&~ -[D_t^8,\rho](\rho^{-1} B\cdot\nab u)  - \eps [a B, D_t^8 ] D_t\qc + \eps a B D_t(D_t^{7} u_i\p_i \qc )  + \eps^2 [a B B_j, D_t^8 ] D_t B_j  .
				\end{aligned}
			\end{equation*}
			Multiplying \eqref{Equ_AGU} by $\varepsilon^{16}(\QQ^8,\UU^8,D_t^8 B)$, integrating over $\Omega$ and integrating by parts, one can see that
			\begin{align}
				&\frac{\varepsilon^{16}}{2} \ddt\int_{\Omega}\left( a|\QQ^8-B\cdot D_t^8 B|^2 +\rho|\UU^8|^2  +|D_t^8 B|^2\right) \dx \notag\\
				=&~\frac{\varepsilon^{16}}{2}\int_{\Omega}\left( (\partial_t a+ \nabla\cdot(a u))|\QQ^8-B\cdot D_t^8 B|^2 +\nabla\cdot u| D_t^8 B|^2  \right)\dx\notag\\
				&~-\varepsilon^{16}\int_{\Omega} a (\eps D_t B \cdot D_t^8 B)( \QQ^8-B\cdot D_t^8 B ) \dd x\notag\\
				&~+\varepsilon^{16}\int_{\Omega} (   (B\cdot\nab) D_t^8 B    )\cdot \UU^8 + ((B\cdot\nabla)\UU^8)\cdot D_t^8 B  \dx \notag\\
				&~+\varepsilon^{16}\int_{\Omega} \left( \QQ^8 \RR_q^8 + \UU^8\cdot \RR_u^8+ D_t^8 B\cdot \RR_B^8\cdot D_t^8 B \right)\dx\notag\\
				&~ + 8\eps^{15}\int_{\Omega} \UU^8 \cdot \nab u_i\,\p_i D_t^7 \qc \dd x \notag\\
				&~ - 9\eps^{16}\int_{\Omega} (B\cdot\nab)(D_t^7 u_i\p_i u) \cdot D_t^8 B \dd x \notag\\
				:=&~\mathcal{G}_1'+\mathcal{G}_2'+\mathcal{G}_3'+\mathcal{G}_4'+\mathcal{G}_5'+\mathcal{G}_6'.\label{energy estimate_Tan8}
			\end{align}
			Since the estimates for $\mathcal{G}_i'$ $(i=1,2,3,4)$ can be obtained using methods similar those in \eqref{Estimate_G1} -- \eqref{Estimate_G4}, we omit the details of the proof and directly present the following conclusions:
			\begin{equation*}
				\mathcal{G}_3'=0,\quad \mathcal{G}_1'+\mathcal{G}_2'+\mathcal{G}_4' \le P(E(t)).
			\end{equation*}
			
			Next, we control the terms $\mathcal{G}_5'$ and $\mathcal{G}_6'$ which cannot be directly controlled by $P(E(t))$ due to the appearance of $\p D_t^7$. By invoking the momentum equation
			\begin{equation*}
				-\eps^{-1}\nab \qc= \rho D_t u - B\cdot \nab B,
			\end{equation*}
			and the magnetic field equation
			\begin{equation*}
				D_t(\rho^{-1}B)= \rho^{-1}B\cdot \nab u,
			\end{equation*}
			$\mathcal{G}_5'$ is rewritten as follows:
			\begin{align}
				\mathcal{G}_5' =&~8\eps^{15}\int_{\Omega} \UU^8 \cdot\nab u_i \, D_t^7 \p_i \qc \dx \underbrace{+ 8\eps^{15}\int_{\Omega} \UU^8 \cdot\nab u_i\, [\p_i, D_t^7] \qc \dx}_{:=\mathcal{G}_{51}'}\notag\\
				=&~8\eps^{16}\int_{\Omega} \UU^8 \cdot\nab u_i \, D_t^7 (B\cdot \nab B_i) \dx \underbrace{-8\eps^{16}\int_{\Omega} \UU^8 \cdot\nab u_i \, D_t^7 (\rho D_t u_i) \dx }_{:=\mathcal{G}_{52}'} + \mathcal{G}_{51}'\notag\\
				=&~8\eps^{16}\int_{\Omega} \UU^8 \cdot\nab u_i \, (B\cdot \nab )(D_t^7 B_i) \dx \underbrace{+8\eps^{16}\int_{\Omega} \UU^8 \cdot\nab u_i \, [D_t^7,\rho] (\rho^{-1}B\cdot\nab B_i) \dx }_{:=\mathcal{G}_{53}'} + \mathcal{G}_{51}'+\mathcal{G}_{52}'
				\end{align}Then we integrate by parts in $\bp$, insert the concrete for of the Alinhac good unknown $\mathbf{U}^8$, use $D_t(\rho^{-1}B)=\rbp u$ and $[D_t,\rbp]=0$ to get
				\begin{align}
				\mathcal{G}_5' =&~-8\eps^{16}\int_{\Omega} (B\cdot \nab)( D_t^8 u_j- 9 D_t^7 u_k\p_k u_j )  \p_j u_i D_t^7 B_i \dx \underbrace{-8\eps^{16}\int_{\Omega} \UU^8_j (B\cdot \nab)(\p_j u_i )   D_t^7 B_i \dx}_{:=\mathcal{G}_{54}'}\notag\\
				&~+ \mathcal{G}_{51}'+\mathcal{G}_{52}'+\mathcal{G}_{53}' \notag\\
				=&~-8\eps^{16}\int_{\Omega} \rho D_t^8 (\rho^{-1} B\cdot\nab u_j) \p_j u_i D_t^7 B_i \dd x + 72 \eps^{16}\int_{\Omega} \rho D_t^7(\rho^{-1} B\cdot\nab u_k) \p_k u_j \p_j u_i D_t^7 B_i \dd x \notag\\
				& ~\underbrace{+   72 \eps^{16}\int_{\Omega} (B\cdot\nab)(\p_k u_j) D_t^7 u_k \p_j u_i D_t^7 B_i \dd x}_{:=\mathcal{G}_{55}'} +\sum_{\ell=1}^{4}\mathcal{G}_{5\ell}'\notag\\
				=&~-8\eps^{16}\int_{\Omega} \p_j u_i\, D_t^9 (\rho^{-1}B_j)\,  D_t^7 B_i \dd x + \underbrace{72 \eps^{16}\int_{\Omega} \rho\p_k u_j\,\p_j u_i\, D_t^8(\rho^{-1} B_k)  D_t^7 B_i \dd x}_{:=\mathcal{G}_{56}'} +\sum_{\ell=1}^{5}\mathcal{G}_{5\ell}'.
				\end{align}
				We then integrate $D_t$ by parts to get
				\begin{align}
				\mathcal{G}_5' =&~ \underbrace{- 8 \eps^{16} \ddt \int_{\Om} \p_j u_i\, D_t^8 (\rho^{-1} B_j)\, D_t^7 B_i \dd x}_{:=\mathcal{G}_5''} \notag\\
				&\underbrace{+8\eps^{16} \int_{\Om} \p_j u_i\, D_t^8 (\rho^{-1} B_j)\,D_t^8 B_i \dd x +8\eps^{16} \int_{\Om} D_t(\p_j u_i) \,D_t^8 (\rho^{-1} B_j) \,D_t^7 B_i \dd x }_{:=\mathcal{G}_{57}'}+\sum_{\ell=1}^{6}\mathcal{G}_{5\ell}'.
			\end{align}
			 By utilizing standard product estimates, it is straightforward to check that
			\begin{equation}
				\sum_{\ell=1}^{7}\mathcal{G}_{5\ell}'\le P(E(t)),
			\end{equation}
			and 
			\begin{align}
				\int_0^t \mathcal{G}_{5}''(\tau) \dd \tau =&~ - \left(8 \eps^{16}  \int_{\Om} \p_j u_i D_t^8 (\rho^{-1} B_j) D_t^7 B_i \dd x\right) \Big|_0^t \notag\\
				\le&~ P(E(0)) +\eta\|\eps^8 D_t^8 (\rho^{-1} B_j)\|_0^2 + C_{\eta} \eps^2 \|\p_j u_i\|_{\infty}^2\|\eps^7 D_t^7 B_i\|_0^2  \notag\\
				\lesssim &~ P(E(0)) +\eta\|\eps^8 D_t^8 (\rho^{-1} B)\|_0^2 +\eps^2 \|u\|_{3}^2\|\eps^7 D_t^7 B\|_0^2  ,\quad \eta<1/2.
			\end{align} The last term is controlled with the help of Jensen's inequality and the Reynolds transport formula (Lemma \ref{lem transport})
			\begin{equation}
			\begin{aligned}
			&~~~ \varepsilon^2\|u\|_3^2\left\|\varepsilon^7 D_t^{7} B\right\|_0^2 \\
			& \lesssim\|u\|_3^2\left(\varepsilon^2\left\|\varepsilon^7 D_t^7 B(0, \cdot)\right\|_0^2+\int_0^t\left\|\left(\varepsilon D_t\right)^8 B(\tau, \cdot)\right\|_0^2 \dtau\right) \\
			&\lesssim \varepsilon^2\|u\|_3^2 E_7(0)+E_4(t) \int_0^t P\left(E_7(\tau)\right) \dtau \text {. } \\
			& \lesssim E_7(0)\left(\varepsilon^2\left\|u_0\right\|_3^2+\int_0^t \| \varepsilon D_t u(\tau, \cdot)\|_3^2 \dtau\right)+E_4(t) \int_0^t P\left(E_7(\tau)\right) \dtau\\
			& \lesssim P(E(0))+\int_0^t P\left(E_4(\tau)\right) \dtau+E_4(t) \int_0^t P\left(E_7(\tau)\right) \dtau.
			\end{aligned}
			\end{equation} Also the term $\|\eps^8 D_t^8 (\rho^{-1} B)\|_0^2$ can be converted to $\|\eps^8 D_t^8  B\|_0^2$ plus controllable terms
			\[
			\|\eps^8 D_t^8 (\rho^{-1} B)\|_0^2 \lesssim \|\eps^8 D_t^8 B\|_0^2  + P(E(0))+ E(t)\int_0^t P(E(\tau))\dtau.
			\]In fact, the leading-order term in $\eps^8[D_t^8,\rho^{-1}]B$ is $(\eps D_t)^8(\rho^{-1}) B$. Since $D_t(\rho^{-1})=-\rho^{-2}D_t\rho=-\rho a \eps D_t q$ (recall $D_t S=0$) produces one more $\eps$ weight, we can move this extra weight to $B$ and control this term similarly as that of $\varepsilon^2\|u\|_3^2\left\|\varepsilon^7 D_t^{7} B\right\|_0^2$.
			
			Therefore, $\mathcal{G}_5'$ is controlled as following:
			\begin{equation}
				\int_{0}^{t}  \mathcal{G}_5'(\tau) \dd \tau \le \eta\|\eps^8 D_t^8 B\|_0^2+ P(E(0))  +E(t)\int_0^t P(E(\tau)) \dd \tau,\quad \eta<1/2.
			\end{equation}
			
			Similarly, one obtains that
			\begin{align}
				\mathcal{G}_6'=&~- 9\eps^{16}\int_{\Omega} \rho D_t^7(\rho^{-1} B\cdot\nab u_i)\p_i u\cdot D_t^8 B \dd x - 9\eps^{16}\int_{\Omega}   \left(( B\cdot\nab)(\p_i u ) D_t^7 u_i + [\rho, D_t^7](\rho^{-1}B\cdot\nab u_i) \p_i u  \right)\cdot D_t^8 B \dd x\notag\\
				\le& ~ - 9\eps^{16}\int_{\Omega} \rho D_t^8(\rho^{-1} B_i)\p_i u\cdot D_t^8 B \dd x +P(E(t))\le P(E(t)).
			\end{align}
			
			Substituting the estimates for $\mathcal{G}_1'$ - $\mathcal{G}_6'$ into \eqref{energy estimate_Tan8}, we obtain that
			\begin{align}
				\|\eps^8 (\QQ^8,\UU^8,D_t^8 B)\|_0^2 \le \eta\|\eps^8 D_t^8 B\|_0^2+ P(E(0)) + E(t)\int_0^t P(E(\tau)) \dd \tau.
			\end{align}
			The control of $(\QQ^8,\UU^8)$ immediately leads to the control of $(D_t^8 q, D_t^8 u)$. For example, we have
			\begin{align*}
			&\|\eps^8 D_t^8 q\|_0^2\lesssim\|\eps^8\QQ^8\|_0^2+\|(\eps^7 D_t^7 u)(\eps\p q)\|_0^2\\
			\lesssim&~\|\eps^8\QQ^8\|_0^2+\|\eps^7 D_t^7 u\|_0^2\left(\|\eps q_0\|_{3}^2+\int_0^t \|\eps D_t q(\tau,\cdot)\|_{3}^2\dtau\right)\\
			\lesssim&~P(E(0))+\int_0^t P(E(\tau)) \dd \tau,
			\end{align*}where we note that the control of $\eps^7 D_t^7 u$ is also used in the above inequality. Similar estimate also holds for $u$, so we conclude that
			\begin{equation}
				\|\eps^8 D_t^8(q,u,B)\|_0^2 \lesssim \eta\|\eps^8 D_t^8 B\|_0^2 + P(E(0)) + E(t)\int_0^t P(E(\tau)) \dd \tau,\q\q\forall\eta\in(0,\frac12).
			\end{equation} Choosing $\eta>0$ suitably small such that the $\eta$-term is absorbed by the left side, we obtain that the desired estimates and the proof of Proposition \ref{Lem_TangentialEstimates} is completed.	
		
		\subsection{Reduction of normal derivatives}\label{sect reduce}
		This part is devoted to the reduction of normal derivatives. Invoking Lemma \ref{hodgeTT}, we obtain the following reduction for $E_4(t)$:
		\begin{align}
			\|(\eps D_t)^k u\|_{4-k}^2\lesssim &~ \|(\eps D_t)^k u\|_0^2+ \|\nab\times(\eps D_t)^k u\|_{3-k}^2 + \|\nab\cdot(\eps D_t)^k u\|_{3-k}^2\notag\\
			\lesssim&~ \|(\eps D_t)^k u\|_0^2 + \|(\eps D_t)^k\nab\times u\|_{3-k}^2 + \|(\eps D_t)^k\nab\cdot u\|_{3-k}^2 \notag\\
			\label{divcurl E4u}&+ \|[(\eps D_t)^k,\nab\times] u\|_{3-k}^2 +\|[(\eps D_t)^k,\nab\cdot] u\|_{3-k}^2, \\
			\|(\eps D_t)^k B\|_{4-k}^2\lesssim & ~\|(\eps D_t)^k B\|_0^2 + \|\nab\times(\eps D_t)^k B\|_{3-k}^2 + \|\nab\cdot (\eps D_t )^k B\|_{3-k}^2\notag\\
			\lesssim&~ \|(\eps D_t)^k B\|_0^2 + \|(\eps D_t)^k\nab\times B\|_{3-k}^2 \notag\\
			\label{divcurl E4B}&+ \|[(\eps D_t)^k,\nab\times] B\|_{3-k}^2 +\|[(\eps D_t)^k,\nab\cdot] B\|_{3-k}^2,
		\end{align} where the boundary part $|(\eps D_t)^k u\cdot N|_{3.5-k}^2$ vanishes thanks to the slip condition $u_3|_{\Sigma}=0$.
		Similarly, for $E_5(t), E_6(t)$ and $E_7(t)$, we have for $1\leq l\leq 3$ and $0\leq k\leq 3-l$ that
		\begin{align}
			\|(\eps D_t)^{k+2l} u\|_{4-k-l}^2\lesssim&~ \|(\eps D_t)^{k+2l} u\|_0^2 + \|(\eps D_t)^{k+2l}\nab\times u\|_{3-k-l}^2 + \|(\eps D_t)^{k+2l}\nab\cdot u\|_{3-k-l}^2 \notag\\
			\label{divcurl E5u}&+ \|[(\eps D_t)^{k+2l},\nab\times] u\|_{3-k-l}^2 +\|[(\eps D_t)^{k+2l},\nab\cdot] u\|_{3-k-l}^2, \\
			\|(\eps D_t)^{k+2l} B\|_{4-k-l}^2\lesssim&~ \|(\eps D_t)^{k+2l} B\|_0^2 + \|(\eps D_t)^{k+2l}\nab\times B\|_{3-k-l}^2  \notag\\
			\label{divcurl E5B}&+ \|[(\eps D_t)^{k+2l},\nab\times] B\|_{3-k-l}^2 +\|[(\eps D_t)^{k+2l},\nab\cdot] B\|_{3-k-l}^2.
		\end{align} 
		In view of the above inequalities, the $L^2$ norms $\|(\eps D_t)^{k+2l} u\|_0^2$ have been controlled in Proposition \ref{Lem_TangentialEstimates}. In what follows, we control the divergence part together with the pressure $q$ in Section \ref{sect reduction q}, the estimates of voriticty and current in Section \ref{sect curl}. The commutators in \eqref{divcurl E4u}-\eqref{divcurl E5B} can be directly controlled and their estimates are recorded at the end of this section and are proved in Lemma \ref{lem prod comm}. Do note that these commutators are not under any time integrals, and so we cannot merely control them by $P(E(t))$ as in Section \ref{sect tg}.

		\subsubsection{Estimates of vorticity and current}\label{sect curl}
		We start with the control of vorticity $\|(\eps D_t)^{k+2l}\nab\times u\|_{3-k-l}^2$ and current $\|(\eps D_t)^{k+2l}\nab\times B\|_{3-k-l}^2$ for $0\leq l\leq 3$ and $0\leq k\leq 3-l$. We first prove the following estimates.
		\begin{lem}\label{lem curl 0}Define $\rho_0(S):=\rho(0,S)$. Then under the assumption of Theorem \ref{main thm, ill data}, for each $l\in\{0,1,2,3\}$ and $0\leq k\leq 3-l$, we have
		\begin{align}
		&\ddt\left(\left\|(\eps D_t)^{k+2l}\nab\times(\rho_0 u)\right\|_{3-k-l}^2+\left\|(\eps D_t)^{k+2l}\nab\times(\rho_0 B)\right\|_{3-k-l}^2+\left\|\eps B\times\left((\eps D_t)^{k+2l}\nab\times(\rho_0 B)\right)\right\|_{3-k-l}^2\right)\notag\\
		\lesssim&~P\left(\sum_{k=0}^{l} E_{4+k}(t)\right) + E_{4+l+1}(t).
		\end{align}
		\end{lem}
		\begin{rmk}
		It should be noted that the loss of Mach number weights occurs if we analyze the curl part via the evolution equation of vorticity $\nab\times u$, because the spatial derivative of $\rho$ is still of $O(1)$ size with respect to the Mach number $\eps$ and $D_t u$ has $O(1/\eps)$ size. Motivated by \cite{Metivier2001limit}, we rewrite the curl part as 
		\begin{equation}
			\nab\times u= \rho_0^{-1} ( \nab\times(\rho_0 u) -\nab\rho_0 \times u ), \quad \nab\times B= \rho_0^{-1} ( \nab\times(\rho_0 B) -\nab\rho_0 \times B), \quad \rho_0:=\rho(0,S).
		\end{equation} 
		In Lemma \ref{lem curl}, we reproduce the estimates for $\nab\times u$ and $\nab\times B$ by combining the above equality and Lemma \ref{lem curl 0}. The advantange of this substitution is that $\rho_0$ commutes with $D_t$ and so we no longer have the singular term. In fact, the momentum equation $\rho D_t u +\eps^{-1}\nabla q+B\times (\nab\times B) =0  $ can be written as
			\begin{equation}\label{MomentumEqu2}
				D_t(\rho_0 u) -\rho_0 \rho^{-1} B\cdot\nab B=-\nab( \eps^{-1} q+ |B|^2/2 ) + \frac{\rho-\rho_0}{\rho} \nab( \eps^{-1} q + |B|^2/2).
			\end{equation}
			There exists a smooth function $g$ such that
			\[
			\frac{\rho-\rho_0}{\rho}=\eps g(\eps q,S),\quad \sum_{k=0}^{4}\|(\eps D_t)^k g\|_{4-k} \le P(E_4(t)).
			\]
			First, we take $\nab\times$ in \eqref{MomentumEqu2} to get the evolution equation
			\begin{align}\label{curlv eq}
				D_t( \nabla\times(\rho_0 u) ) -\nabla\times\left( \rho_0\rho^{-1} B\cdot\nabla B \right) =[u\cdot\nab,\nabla\times](\rho_0 u) + \nabla g \times\nabla (q+ \eps|B|^2/2),
			\end{align}where we notice that the right side only contains first-order derivatives and does not lose $\eps$ weights. 
		\end{rmk}

%
		We first prove the case when $k=0$. Indeed, the cases for $k=1,2,3$ follow in the same manner and the analysis will be recorded in Lemma \ref{lem_curl_H5} and Lemma \ref{lem_curl_H6}. 
		
		\begin{lem}\label{lem_curl_H4}
			Under the assumptions of Theorem \ref{main thm, ill data}, we have
			\begin{equation}
				\sum_{k=0}^{3}\ddt \left(\|(\eps D_t)^k\nabla\times(\rho_0 u)\|_{3-k}^2 +\|(\eps D_t)^k\nabla\times(\rho_0 B)\|_{3-k}^2 +\left\|\eps B\times\left((\eps D_t)^k\nab\times (\rho_0 B)\right)\right\|_{3-k}^2 \right) \lesssim P(E_4(t))+E_5(t).
			\end{equation}
		\end{lem}
		\begin{proof}
			
			We first treat the case $k=0$. Taking $\p^3$ in \eqref{curlv eq}, we get
			\begin{align}\label{curlv H3 eq}
				D_t(\p^3 \nabla\times(\rho_0 u) ) -\p^3\nabla\times\left( \rho_0 \rho^{-1} B\cdot\nabla B \right)=   \underbrace{\p^3(\text{RHS of }\eqref{curlv eq})+[u\cdot\nab,\p^3](\nab\times (\rho_0 u)) }_{:=\RR_1},
			\end{align}
			where the order of derivatives on the right side does not exceed 4. We start with the estimates of vorticity by
			\begin{align}
				&\frac{1}{2}\ddt \int_{\Omega} |\p^3\nab\times (\rho_0 u)|^2 \dx=\int_{\Omega} (\p_t \p^3 \nab\times (\rho_0 u) )\cdot \p^3 \nab\times (\rho_0 u) \dx  \notag\\
				=&\int_{\Om} D_t \p^3 \nab\times (\rho_0 u) \cdot\p^3\nab\times (\rho_0 u) \dx \underbrace{- \int_{\Om}(u\cdot\nab) \p^3\nab\times (\rho_0 u)\cdot \p^3
					\nab \times (\rho_0 u) \dx }_{:=\mathcal{I}_1}.
			\end{align}
			Then invoking \eqref{curlv H3 eq} gives us
			\begin{align}
				&\int_{\Om}D_t \p^3 \nab\times (\rho_0 u) \cdot\p^3\nab\times (\rho_0 u) \dx \notag\\
				=&~\int_{\Om} \p^3 \nabla\times\left( \rho_0\rho^{-1} B\cdot\nabla B \right) \cdot \p^3 \nab\times (\rho_0 u) \dx  \underbrace{+ \int_{\Om}  \RR_1\cdot\p^3 \nab\times (\rho_0 u)\dx }_{:=\mathcal{I}_2}\notag\\
				=&~\int_{\Om}\left(\rho^{-1}(B\cdot\nab) \p^3 \nab\times (\rho_0 B)\right)\cdot\p^3 \nab\times (\rho_0 u) \dx +\mathcal{I}_2 \notag\\
				&\underbrace{+\int_{\Om} \left( \p^3( \nab(\rho^{-1} B_i)\times\p_i(\rho_0 B) ) + [\p^3,\rho^{-1}B\cdot\nab] \nab\times(\rho_0 B)   \right)\cdot\p^3 \nab\times (\rho_0 u) \dx}_{:=\mathcal{I}_3}  \notag\\
				&\underbrace{-\int_{\Om} \p^3\nabla\times\left(  (\rho^{-1} B\cdot\nab\rho_0) B  \right)\cdot\p^3 \nab\times (\rho_0 u) \dx}_{:=\mathcal{I}_4}.
				\end{align}
				Integrating by parts in $\int_{\Om}\left(\rho^{-1}(B\cdot\nab) \p^3 \nab\times (\rho_0 B)\right)\cdot\p^3 \nab\times (\rho_0 u) \dx$, we obtain
				\begin{align}
				&\int_{\Om}\left(\rho^{-1}(B\cdot\nab) \p^3 \nab\times (\rho_0 B)\right)\cdot\p^3 \nab\times (\rho_0 u) \dx\notag\\
				\overset{B\cdot\nab}{=}&-\int_{\Om}\p^3\nab\times (\rho_0 B) \cdot \p^3\nab\times ((\rho^{-1}B\cdot\nab) (\rho_0u)) \dx\notag\\
				& \underbrace{- \int_{\Om}(\p^3 \nab\times (\rho_0 B)) \cdot  \left( [\rho^{-1}B\cdot\nab, \p^3]\nab\times(\rho_0 u)- \p^3(\nab(\rho^{-1}B_i)\times \p_i(\rho_0 u) ) +\nabla\cdot (\rho^{-1}B)  \p^3 \nabla\times (\rho_0 u)   \right) \dx}_{:=\mathcal{I}_5}\notag\\
				=&-\int_{\Om}(\p^3\nab\times (\rho_0 B)) \cdot \p^3\nab\times ( \rho_0\rho^{-1} B\cdot\nab u ) \dx +\mathcal{I}_5\notag\\
				&\underbrace{-\int_{\Om}  \p^3\nab\times (\rho_0 B) \cdot \p^3\nabla\times\left( (\rho^{-1}B\cdot\nab\rho_0)u  \right) \dx }_{:=\mathcal{I}_6}.
			\end{align}
			Inserting the evolution equation of $B$ to substitute $B\cdot\nab u$, we get
			\begin{align}
				&-\int_{\Om}(\p^3\nab\times (\rho_0 B)) \cdot \p^3\nab\times (\rho_0 \rho^{-1}B\cdot \nab u) \dx \notag \\
				=&-\int_{\Om}\p^3 \nab\times (\rho_0 B) \cdot \p^3 \nab\times (\rho_0\rho^{-1} D_t B) \dx -\int_{\Om}\p^3 \nab\times (\rho_0 B) \cdot \p^3 \nab\times (\rho_0\rho^{-1} B\nab\cdot u) \dx  \notag\\
				=&-\frac{1}{2}\ddt\int_{\Omega} \rho^{-1}|\p^3\nab\times (\rho_0 B)|^2 \dx \underbrace{+\frac{1}{2} \int_{\Om} ( D_t\rho^{-1} + \rho^{-1}\nab\cdot u )|\p^3\nab\times (\rho_0 B)|^2\dx}_{:=\mathcal{I}_7}\notag\\
				&\underbrace{- \int_{\Om} \p^3\nab\times (\rho_0 B) \cdot \left( \nab\times( [\p^3,\rho^{-1} D_t](\rho_0 B)  ) +\nab\rho^{-1}\times\p_t\p^3(\rho_0 B) +\nab(\rho^{-1}u_i)\p_i\p^3(\rho_0 B)  \right) \dx}_{:=\mathcal{I}_8}\notag\\
				&\underbrace{+\int_{\Om}\p^3\nab\times (\rho_0 B) \cdot \p^3\nab\times ( \eps\rho_0\rho^{-1}a B D_t q) \dx}_{:=\mathcal{K}}.
			\end{align}
			Based on the concrete forms of the commutators in Lemma \ref{Prop_AGU}, a straightforward product estimate for $\mathcal{I}_j$ $(j=1,2,3,5,7)$ gives us
			\begin{align}
				\mathcal{I}_1 + \mathcal{I}_2 + \mathcal{I}_3+ \mathcal{I}_5+ \mathcal{I}_7 \le P(E_4(t)).
			\end{align}
			In $\mathcal{I}_4$ and $\mathcal{I}_6$, we find that there are fourth-order derivatives falling on $\rho^{-1}B\cdot\nab \rho_0$. Since $\rho_0=\rho(0,S)$ only depends on $S$, we can use the enhanced ``directional" regularity of $S$ (Corollary \ref{Cor_Entropy}) to control such terms
			\begin{align}
			\mathcal{I}_4+ \mathcal{I}_6 \le P(E_4(t)).
			\end{align}
			For the terms in $\mathcal{I}_8$, there exhibits a time derivative without $\eps$ weight. In fact, this time derivative can be replaced with spatial derivatives by using the equations for the magnetic field and entropy. For example, we obtain that
			\begin{align}
				&~\|\nab\times([\p^3, \rho^{-1}D_t](\rho_0 B)) \|_0 \notag\\
				\lesssim&~ \| \p_j \rho^{-1} \p_t(\rho_0  B)  \|_3 + \|\p_j (\rho^{-1}u_k )\p_k(\rho_0  B) \|_3\notag\\
				\lesssim&~ \| (\eps\p_q \rho^{-1} \p_j q  + \p_S \rho^{-1} \p_j S) \cdot(\rho_0' B \p_t S + \rho_0 \p_t B   ) \|_3 + P(E_4(t))\notag\\
				\lesssim&~ \|\rho_0' B   \p_S \rho^{-1} \p_j S (u\cdot\nab S ) \|_3+\| \rho_0  \p_S \rho^{-1} \p_j S (-u\cdot\nab B+B\cdot \nab u -B\nab\cdot  u ) \|_3+ P(E_4(t))\notag\\
				\le&~ P(E_4(t)),
			\end{align}
			which leads to
			\begin{equation}
				\mathcal{I}_8 \le P(E_4(t)).
			\end{equation}
			
			The crucial term is $\mathcal{K}$ in which the highest-order term has 5-th order derivative and thus cannot be controlled by $E_4(t)$. We first pick up the highest-order part in $\mathcal{K}$:
			\begin{align*}
				\p^3\nab\times(\eps\rho_0\rho^{-1} a B D_t q)=&~\eps\p^3((\nab D_t q)\times (\rho_0\rho^{-1} a B)+D_t q \nab\times (\rho_0\rho^{-1} a B))\\
				=&-\eps \rho_0\rho^{-1} a B\times (\p^3\nab D_t q)\underbrace{-\eps[\p^3, \rho_0\rho^{-1}a B]\times(\nab D_t q)+\eps\p^3(D_t q\nab\times (\rho_0\rho^{-1} a B))}_{:=\RR_2},
			\end{align*}
			where straightforward calculation shows that $\|\RR_2\|_0\leq P(E_4(t))$. For the fifth-order part, the vector identity $(B\cdot \nab)B-\frac12\nab|B|^2=-B\times(\nab\times B)$ motivates us to rewrite the momentum equation as
			\[
			\rho D_t u+B\times(\nab\times B)=-\eps^{-1}\nab q.
			\] So, we commute $D_t$ with $\nab$ and plug this equation into $-\eps \rho_0\rho^{-1} a B\times (\p^3\nab D_t q)$ to get 
			\begin{align}
				&-\eps\rho_0\rho^{-1}a B\times (\p^3\nab D_t q)\notag\\
				=&~-\eps\rho_0\rho^{-1}a B\times(\p^3D_t\nab q)\underbrace{-\eps\rho_0\rho^{-1}a B\times(\p^3(\nab u_j\p_j q))}_{:=\mathcal{R}_3}\notag\\
				= &~\eps^2\rho_0\rho^{-1} a B\times\p^3 D_t (\rho D_t u)+\eps^2\rho_0\rho^{-1} a  B\times \p^3 D_t(B\times(\nab\times B))+\mathcal{R}_3\notag\\
				= &~\eps^2\rho_0 a B\times( \p^3 D_t^2 u) +\underbrace{\eps^2\rho^{-1} a B\times D_t(B\times\p^3 (\nab\times (\rho_0 B)))}_{=\mathcal{K}_1}\notag\\
				&\underbrace{+\eps^2\rho_0\rho^{-1} a B\times\left([\p^3 D_t,\rho] D_t u
					+[\p^3, D_t](\nab\times B)+D_t\left([\p^3,B]\times(\nab\times B) \right) +D_t \left( [\rho_0,\p^3\nab ]\times B\right)\right)}_{:=\RR_4}+\RR_3,
			\end{align} 
			where $\|\RR_3+\RR_4\|_0\leq P(E_4(t))$ can be proved by using the concrete forms of commutators because the order of derivatives does not exceed 4. It should be noted that the time derivative in $D_t\left([\p^3,B\times](\nab\times B)\right)$ does not lead to any loss of $\eps$ weights because there is already an $\eps$ factor in $\RR_3$. Similar argument also applies to the control of $\RR_4$. Now, we analyze the contribution of $\mathcal{K}_1$ in the integral $\mathcal{K}$, presented as below
			\begin{align}
				\io \eps^2\rho^{-1} a (\p^3\nab\times (\rho_0 B)) \cdot \left(B\times D_t\left(B\times\p^3 \nab\times (\rho_0 B)\right)\right)\dx.
			\end{align} We set $\bd{u}=B$, $\bd{v}=D_t\left(B\times\p^3 \nab\times (\rho_0 B)\right)$ and $\bd{w}=\p^3(\nab\times (\rho_0 B))$ in the vector identity $(\bd{u}\times \bd{v})\cdot\bd{w}=-(\bd{u}\times \bd{w})\cdot\bd{v}$ to get
			\begin{align}
				&\io \eps^2\rho^{-1} a (\p^3\nab\times (\rho_0 B)) \cdot \left(B\times D_t\left(B\times\p^3 \nab\times (\rho_0 B)\right)\right)\dx\notag\\
				=& -\io\eps^2\rho^{-1} a D_t\left(B\times(\p^3 \nab\times (\rho_0 B))\right) \cdot \left(B\times(\p^3\nab\times (\rho_0 B))\right)\dx \notag\\
				=&-\frac12\ddt\io \eps^2\rho^{-1}a\left|B\times(\p^3 \nab\times (\rho_0 B))\right|^2\dx +\underbrace{\frac12\io\eps^2(\rho^{-1} a\nab\cdot u+D_t (\rho^{-1}a))\left|B\times(\p^3 \nab\times (\rho_0 B))\right|^2\dx}_{:=\RR_5}.
			\end{align}  The term $\RR_5$ is directly controlled by $ P(E_4(t))$. Therefore, we get the uniform control for $\mathcal{K}$ by using both $E_4$ and $E_5$:
			\begin{align}
				\mathcal{K} + \frac12\ddt\io \eps^2\rho^{-1}a\left|B\times(\p^3 \nab\times (\rho_0 B))\right|^2\dx  \le \|B\|_4\left(P(E_4(t))+\|\rho_0 a B\|_{\infty}\|\eps^2 D_t^2 u\|_3\right)\le P(E_4(t))+E_5(t),
			\end{align}
			which gives us the energy estimate for the case of $k=0$
			\begin{align}
				\ddt \left(\|\nab\times (\rho_0 u)\|_3^2 +\|\nab\times (\rho_0 B)\|_3^2+\left\|\eps B\times(\nab\times (\rho_0 B))\right\|_3^2\right)\leq P(E_4(t))+E_5(t).
			\end{align}

			Similarly, we can prove the same conclusion for $\p^{\alpha}D_t^k$ with $k+|\alpha|=3$ by replacing $\p^3$ with $\p^{\alpha}(\eps D_t)^k$. Indeed, the highest order of derivatives in the above commutators does not exceed 4-th order, and there is no loss of Mach number weight because none of the above steps creates negative power of Mach number. The fifth-order term can also be analyzed in the same way:  
			\[
			\eps \rho_0\rho^{-1} a B\times (\p^\alpha(\eps D_t)^k\nab D_t q)\lleq -\rho_0 a  B\times\p^{\alpha}(\eps D_t)^{k+2} u-\eps^2 \rho^{-1} a B\times D_t\left(B\times(\p^{\alpha}(\eps D_t)^k \nab\times (\rho_0 B))\right),
			\]
			where the $L^2$ norms of the omitted lower-order terms are controlled by $P(E_4(t))$. The contribution of the last term in the above equality is again the energy term $-\frac12\ddt\io  \eps^2\rho^{-1} a\left|B\times(\p^\alpha(\eps D_t)^k \nab\times B)\right|^2\dx$ if we replace $\p^3$ above by $\p^{\alpha}(\eps D_t)^k$. Hence, we obtain the following estimates for vorticity and current in $E_4(t)$
			\begin{align}
				\sum_{k=0}^{3}\ddt \left(\left\|(\eps D_t)^k\nab\times (\rho_0 u)\right\|_{3-k}^2 +\left\|(\eps D_t)^k\nab\times(\rho_0 B)\right\|_{3-k}^2+\left\|\eps B\times\left((\eps D_t)^k\nab\times B\right) \right\|_{3-k}^2\right) \leq P(E_4(t))+E_5(t).
			\end{align}
		
		\end{proof}

		Next, we turn to prove Lemma \ref{lem curl 0} for $1\leq l\leq 3$. The proof is parallel to Lemma \ref{lem_curl_H4} and we mostly focus on the case of $l=1$.
		\begin{lem}\label{lem_curl_H5}
			It holds that
			\begin{equation}
			\begin{aligned}
				&\sum_{k=0}^{2}\ddt\left( \left\|(\eps D_t)^{k+2}\nab\times  (\rho_0 u)\right\|_{2-k}^2 + \left\|(\eps D_t)^{k+2}\nab\times  (\rho_0 B)\right\|_{2-k}^2 +  \left\|\eps B\times\left((\eps D_t)^{k+2}\nab\times  (\rho_0 B)\right)\right\|_{2-k}^2 \right)\\
				\le&~ P(E_4(t),E_5(t))+E_6(t).
			\end{aligned}
			\end{equation}
		\end{lem}

		\begin{proof} For the case of $k=0$, we take $\p^2(\eps D_t)^2$ in \eqref{curlv eq} to get 
			\begin{align}\label{curlv H5 eq}
				&D_t(\p^2(\eps D_t)^2 \nabla\times(\rho_0 u) ) -\p^2(\eps D_t)^2\nabla\times\left( \rho_0 \rho^{-1} B\cdot\nabla B \right)\notag\\
				=&   \underbrace{\p^2(\eps D_t)^2(\text{RHS of }\eqref{curlv eq})+[u\cdot\nab,\p^2]((\eps D_t)^2\nab\times (\rho_0 u)) }_{:=\RR_1'}.
			\end{align}
			Now, standard $L^2$-type estimate yields that
			\begin{align}
				&\frac{1}{2}\ddt \int_{\Omega} |\p^2(\eps D_t)^2\nab\times (\rho_0 u)|^2 \dx\notag\\
				=&\int_{\Omega} (\p_t \p^2(\eps D_t)^2 \nab\times (\rho_0 u) )\cdot \p^2 (\eps D_t)^2 \nab\times (\rho_0 u) \dx  \notag\\
				=&\int_{\Om} D_t \p^2(\eps D_t)^2 \nab\times (\rho_0 u) \cdot \p^2(\eps D_t)^2\nab\times (\rho_0 u) \dx \underbrace{- \int_{\Om}(u\cdot\nab)  \p^2(\eps D_t)^2\nab\times (\rho_0 u)\cdot  \p^2(\eps D_t)^2
					\nab \times (\rho_0 u) \dx }_{:=\mathcal{I}_1'}.
			\end{align}
			
			Then invoking \eqref{curlv H5 eq} give us the following
			\begin{align}
				&\int_{\Om}D_t  \p^2(\eps D_t)^2 \nab\times (\rho_0 u) \cdot \p^2(\eps D_t)^2\nab\times (\rho_0 u) \dx \notag\\
				=&\int_{\Om}  \p^2(\eps D_t)^2 \nabla\times\left( \rho_0\rho^{-1} B\cdot\nabla B \right) \cdot  \p^2(\eps D_t)^2 \nab\times (\rho_0 u) \dx  \underbrace{+ \int_{\Om}  \RR_1'\cdot \p^2(\eps D_t)^2 \nab\times (\rho_0 u)\dx }_{:=\mathcal{I}_2'}\notag\\
				=&\int_{\Om}   \p^2(\eps D_t)^2 \nab\times (\rho_0 u) \cdot \p^2(\eps D_t)^2 \left( (\rho^{-1}B\cdot\nab) \nab\times (\rho_0 B)  \right)\dx\notag\\ 
				&\underbrace{-\int_{\Om}   \p^2(\eps D_t)^2 \nab\times (\rho_0 u) \cdot \p^2(\eps D_t)^2 \nab\times\left(  (  \rho^{-1}B\cdot\nab \rho_0)B  \right)\dx}_{:=\mathcal{I}_3'}\notag\\ 
				&\underbrace{+\int_{\Om}   \p^2(\eps D_t)^2 \nab\times (\rho_0 u) \cdot\p^2(\eps D_t)^2 \left(\nab(\rho^{-1} B_i)\times\p_i(\rho_0 B)\right)\dx}_{:=\mathcal{I}_4'}  +\mathcal{I}_2'\notag\\ 
				=&  \int_{\Om}   \p^2(\eps D_t)^2 \nab\times (\rho_0 u) \cdot (\rho^{-1}B\cdot\nab)( \p^2 (\eps D_t)^2 \nab\times(\rho_0 B) ) \dx  \notag\\
				&\underbrace{+  \int_{\Om}   \p^2(\eps D_t)^2 \nab\times (\rho_0 u) \cdot [\p^2,\rho^{-1}B\cdot\nab](  (\eps D_t)^2 \nab\times(\rho_0 B) ) \dx }_{:=\mathcal{I}_5'} \notag\\
				&+  \int_{\Om}   \p^2(\eps D_t)^2 \nab\times (\rho_0 u) \cdot \underbrace{\p^2 \left( [(\eps D_t)^2,\rho^{-1}B\cdot\nab] \nab\times(\rho_0 B) \right)}_{=0}  \dx +\sum_{j=1}^{4} \mathcal{I}_j',
			\end{align}
			where we have used the fact $[\rho^{-1}B\cdot\nab,D_t]=0$. Then we integrate $B\cdot\nab$ by parts and invoke the evolution equation of $B$ to get 
			\begin{align}
				&  \int_{\Om}   \p^2(\eps D_t)^2 \nab\times (\rho_0 u) \cdot \left[(\rho^{-1}B\cdot\nab)( \p^2 (\eps D_t)^2 \nab\times(\rho_0 B) )\right] \dx  \notag\\
				=&-\int_{\Om}\p^2(\eps D_t)^2\nab\times (\rho_0 B) \cdot \left[\p^2(\eps D_t)^2\nab\times ((\rho^{-1}B\cdot\nab) (\rho_0u))\right] \dx  \notag\\
				&\underbrace{- \int_{\Om}\p^2(\eps D_t)^2\nab\times (\rho_0 B) \cdot  \left(  [\rho^{-1}B\cdot\nab,\p^2]( (\eps D_t)^2\nab\times (\rho_0 B) ) - \p^2 (\eps D_t)^2\left(\nab(\rho^{-1}B_i)\times\p_i(\rho_0 u)\right)\right) \dx }_{:=\mathcal{I}_5'}\notag\\
				&- \int_{\Om}\p^2(\eps D_t)^2\nab\times (\rho_0 B) \cdot \underbrace{\p^2 \left(  [\rho^{-1}B\cdot\nab,(\eps D_t)^2] \nab\times(\rho_0 u) \right)}_{=0} \dx \notag\\
				& \underbrace{- \int_{\Om}\p^2(\eps D_t)^2\nab\times (\rho_0 B) \cdot  \left( \nabla\cdot (\rho^{-1}B)  \p^2 (\eps D_t)^2 \nabla\times (\rho_0 u)   \right) \dx}_{:=\mathcal{I}_6'}\notag\\
				=&-\int_{\Om}\p^2(\eps D_t)^2\nab\times (\rho_0 B)\cdot \p^2(\eps D_t)^2\nab\times ( \rho_0\rho^{-1} B\cdot\nab u ) \dx +\mathcal{I}_5' +\mathcal{I}_6'\notag\\
				&\underbrace{-\int_{\Om}  \p^2(\eps D_t)^2\nab\times (\rho_0 B) \cdot \p^2(\eps D_t)^2\nab\times\left((\rho^{-1}B\cdot\nab \rho_0)u\right) \dx }_{:=\mathcal{I}_7'},\notag\\
				=&-\frac{1}{2} \int_{\Om} \rho^{-1}\left|  \p^2(\eps D_t)^2\nab\times (\rho_0 B)   \right|^2 \dx \underbrace{+\frac{1}{2} \int_{\Om} ( \rho^{-1}\nab\cdot u + D_t\rho^{-1} )\left|  \p^2(\eps D_t)^2\nab\times (\rho_0 B)   \right|^2 \dx}_{:=\mathcal{I}_8'}\notag\\
				&\underbrace{-\int_{\Om}\p^2(\eps D_t)^2\nab\times (\rho_0 B)\cdot\p^2(\eps D_t)^2 \left(\nab\rho^{-1}\times \p_t(\rho_0 B)+ \nab(\rho^{-1} u_i)\times\p_i(\rho_0 B)\right) \dx}_{:=\mathcal{I}_9'}\notag\\
				&\underbrace{-\int_{\Om}\p^2(\eps D_t)^2\nab\times (\rho_0 B)\cdot  [\p^2(\eps D_t)^2,\rho^{-1}D_t] \nab\times (\rho_0 B)\dx}_{:=\mathcal{I}_{10}'}\notag\\
				&\underbrace{+\int_{\Om}\p^2(\eps D_t)^2\nab\times (\rho_0 B)\cdot \p^2(\eps D_t)^2\nab\times ( \eps \rho_0\rho^{-1}a B D_t q) \dx}_{:=\mathcal{K}'}
			\end{align}
			A straightforward product estimate for $\mathcal{I}_j'$ $(1\le j\le 10)$ gives us
			\begin{equation}
				\sum_{j=1}^{10}\mathcal{I}_j'\le P(E_4(t),E_5(t)),
			\end{equation}where the control of $\mathcal{I}_3'$ and $\mathcal{I}_7'$ again requires the enhanced regularity of $S$ in the direction of $B/\rho$ as shown in Corollary \ref{Cor_Entropy}.

			For $\mathcal{K}'$, we shall control $ \p^2(\eps D_t)^2\nab\times(\eps \rho_0\rho^{-1}a B D_t q)$. We have
			\begin{align*}
				&\p^2(\eps D_t)^2\nab\times(\eps\rho_0\rho^{-1} a B D_t q)\\
				=&-\eps \rho_0\rho^{-1} a B\times (\p^2(\eps D_t)^2\nab D_t q)\underbrace{-\eps[\p^2(\eps D_t)^2, \rho_0\rho^{-1}a B]\times(\nab D_t q)+\eps\p^2(\eps D_t)^2(D_t q\nab\times (\rho_0\rho^{-1} a B))}_{:=\RR_2'}\\
				=&~-\rho_0\rho^{-1}a B\times(\p^2(\eps D_t)^3\nab q)\underbrace{-\eps\rho_0\rho^{-1}a B\times(\p^2(\eps D_t)^2(\nab u_j\p_j q))}_{:=\mathcal{R}_3'} +\RR_2'\notag\\
				= &~\rho_0\rho^{-1} a B\times\p^2(\eps D_t)^3 (\eps\rho D_t u)+\rho_0\rho^{-1} a  B\times \p^2(\eps D_t)^3(\eps B\times(\nab\times B))++\RR_2'+\mathcal{R}_3'\notag\\
				= &~\rho_0 a B\times( \p^2(\eps D_t)^4 u) +\underbrace{\eps^2\rho^{-1} a B\times D_t(B\times\p^2(\eps D_t)^2 (\nab\times (\rho_0 B)))}_{=\mathcal{K}_1'}\notag\\
				&\underbrace{+\eps^4\rho_0\rho^{-1} a B\times\left([\p^2 D_t^3,\rho] D_t u
					+[\p^2, D_t](D_t^2\nab\times B)+D_t\left([\p^2 D_t^2,B]\times(\nab\times B) \right) +D_t \left( [\rho_0,\p^2 D_t^2\nab\times ]B\right)\right)}_{:=\RR_4'}\\
				&+\RR_2'+\RR_3'.
			\end{align*}
			Similarly as in the proof of Lemma \ref{lem_curl_H4}, we have
			\begin{equation*}
				\|\RR_2'+\RR_3'+\RR_4'\|_0\le P(E_4(t),E_5(t)).
			\end{equation*}
			
			The analysis of $\mathcal{K}_1'$ is also parallel to the analysis of $\mathcal{K}_1$. 
			\begin{align}
				&\int_{\Om}\left(\p^2(\eps D_t)^2 \nab\times (\rho_0 B)\right)\cdot \left(\eps^2\rho^{-1} a B\times D_t(B\times\p^2(\eps D_t)^2 (\nab\times (\rho_0 B))) \right)\dx\no\\
				=&-\frac12\ddt\io\eps^2 \rho^{-1}a\left|B\times\p^2(\eps D_t)^2 \nab\times (\rho_0 B)\right|^2\dx+\frac12\io(\rho^{-1}a \nab\cdot u + D_t(\rho^{-1}u))\left|B\times\p^2(\eps D_t)^2 \nab\times (\rho_0 B)\right|^2\dx\no\\
				\le&-\frac12\ddt\io\eps^2 \rho^{-1}a\left|B\times\p^2(\eps D_t)^2 \nab\times (\rho_0 B)\right|^2\dx+P(E_4(t),E_5(t)),
			\end{align}
			and thus
			\begin{equation}
				\mathcal{K}'+\frac12\ddt\io\eps^2 \rho^{-1}a\left|B\times\p^2(\eps D_t)^2 \nab\times (\rho_0 B)\right|^2\dx\le P(E_4(t))\|(\eps D_t)^4 u\|_2+ P(E_4(t),E_5(t)).
			\end{equation}
			So we obtain the vorticity estimates for $k=0$ 
			\begin{equation}
				\begin{aligned}
					\ddt\left(\left\|(\eps D_t)^2\nab\times(\rho_0 u)\right\|_2^2+ \left\|(\eps D_t)^2\nab\times(\rho_0 B)\right\|_2^2+\left\|\eps B\times ((\eps D_t)^2 \nab\times(\rho_0 B))\right\|_2^2\right)\le P(E_4(t),E_5(t))+E_6(t).
				\end{aligned}
			\end{equation}

			For $k=1,2$, we get the following estimates in the same way as in the proof of Lemma \ref{lem_curl_H4}
			\begin{align}
				&\ddt\left(\left\|(\eps D_t)^{k+2}\nab\times(\rho_0 u)\right\|_{2-k}^2+ \left\|(\eps D_t)^{k+2}\nab\times(\rho_0 B)\right\|_{2-k}^2+\left\|\eps B\times ((\eps D_t)^{k+2} \nab\times(\rho_0 B))\right\|_{2-k}^2\right)\no\\
				\le&~P(E_4(t),E_5(t))+E_6(t).
			\end{align}
			The proof of Lemma \ref{lem_curl_H5} is completed.\end{proof}
For the case $l=2,3$, we can also obtain the desired estimates by mimicing the proof of Lemma \ref{lem_curl_H4} and Lemma \ref{lem_curl_H5}.
		\begin{lem}\label{lem_curl_H6}
			It holds that
			\begin{align}
				&\sum_{k=0,1} \ddt \left(  \left\|(\eps D_t)^{k+4}\nab\times (\rho_0 u)\right\|_{1-k}^2+ \left\|(\eps D_t)^{k+4}\nab\times (\rho_0 B)\right\|_{1-k}^2 +  \left\|\eps B\times\left((\eps D_t)^{k+4}\nab\times (\rho_0 B)\right)\right\|_{1-k}^2\right) \notag\\
				\le&~ P(E_4(t),E_5(t),E_6(t))+E_7(t),\\
				&\ddt \left(  \left\|(\eps D_t)^{6}\nab\times (\rho_0 u)\right\|_{0}^2+ \left\|(\eps D_t)^{6}\nab\times (\rho_0 B)\right\|_{0}^2 +  \left\|\eps B\times\left((\eps D_t)^{6}\nab\times (\rho_0 B)\right)\right\|_{0}^2\right) \notag\\
				\le&~ P(E_4(t),E_5(t),E_6(t),E_7(t))+E_8(t).
			\end{align}
		\end{lem}
		Summarizing Lemma \ref{lem_curl_H4}-Lemma \ref{lem_curl_H6}, we complete the proof of Lemma \ref{lem curl 0}. We now need to recover the estimates for $\nab\times u$ and $\nab\times B$ from the control in Lemma \ref{lem curl 0}.
		
		\begin{lem}[Estimates of vorticity and current]\label{lem curl}
		For $0\leq l\leq 3$ and $0\leq k\leq 3-l$, the following estimates for the voricity and the current hold uniformly in $\eps$:
		\begin{align}
		&\left\|(\eps D_t)^{k+2l}\nab\times u\right\|_{3-k-l}^2 + \left\|(\eps D_t)^{k+2l}\nab\times B\right\|_{3-k-l}^2 \notag\\
		\label{curl estimate} \lesssim&~ P\left(E(0)\right) + P(E(t)) \int_0^tP(E(\tau))\dtau +\sum_{k_0\leq k\atop l_0\leq l} \left(\|(\eps D_t)^{k_0+2l_0}u\|_{3-k_0-l_0}^2+\|(\eps D_t)^{k_0+2l_0}B\|_{3-k_0-l_0}^2\right).
		\end{align}	
		\end{lem}
		\begin{proof}
		It suffices to prove the estimates for $u$ and the same argument applies to $B$. For fixed $k,l\in\Z$ satisfying $0\leq l\leq 3$ and $0\leq k\leq 3-l$, we have
		\[
		\nab\times u = \rho_0^{-1} \left(\nab\times(\rho_0 u )- \nab\rho_0\times u\right),~~\nab\times B = \rho_0^{-1} \left(\nab\times(\rho_0 B)- \nab\rho_0\times B\right),
		\]and thus
		\begin{align*}
		(\eps D_t)^{k+2l}\nab\times u =&~ (\eps D_t)^{k+2l}\left[\rho_0^{-1} \left(\nab\times(\rho_0 u )- \nab\rho_0\times u\right)\right]\\
		=&~\rho_0^{-1}(\eps D_t)^{k+2l}\left(\nab\times(\rho_0 u )\right) - \rho_0^{-1}\sum_{j=0}^{k+2l}\binom{k+2l}{j} (\eps D_t)^j\nab\rho_0 \times (\eps D_t)^{k+2l-j} u\\
		=&~\rho_0^{-1}(\eps D_t)^{k+2l}\left(\nab\times(\rho_0 u )\right) - \rho_0^{-1}\sum_{j=0}^{k+2l}\binom{k+2l}{j}\left([(\eps D_t)^j,\nab]\rho_0\right) \times (\eps D_t)^{k+2l-j} u
		\end{align*}where we use the fact $D_t\rho_0=0$. The first term is controlled by using Lemma \ref{lem curl 0} and Corollary \ref{Cor_Entropy}
		\begin{align}
		 \left\|\rho_0^{-1}(\eps D_t)^{k+2l}\left(\nab\times(\rho_0 u )\right)\right\|_{3-k-l}^2\leq P\left(\sum_{j=0}^{l}E_{4+j}(0)\right) + \int_0^tP\left(\sum_{j=0}^{l}E_{4+j}(\tau)\right)+ E_{4+l+1}(\tau)\dtau.
		\end{align} The second term is controlled by using the conclusion of Lemma \ref{lem prod comm}
		\begin{align}
		&\left\|\left([(\eps D_t)^j,\nab]\rho_0\right) \times (\eps D_t)^{k+2l-j} u\right\|_{3-k-l}^2\notag\\
		\lesssim&~ P\left(E(0)\right) + P(E(t))\int_0^tP(E(\tau))\dtau +\sum_{k_0\leq k\atop l_0\leq l} \|(\eps D_t)^{k_0+2l_0}u\|_{3-k_0-l_0}^2.
		\end{align}
		\end{proof}
		\begin{rmk}
		The estimate \eqref{curl estimate} in this proposition is not the final version that we use to close the uniform energy estimates for $E(t)$ because $\|\cdot\|_{3-k-l}$ norm still contains normal derivatives of $u$ and $B$. The good thing is that the remaining terms to be controlled, namely the last term in \eqref{curl estimate}, only belong to $H_*^7(\Om)$ and the number of $\eps$ weights remains the same as before. At the end of this section, we will repeatedly apply the div-curl analysis to these terms such that they can finally be controlled by $P(E(0))+\int_0^t P(E(\tau))\dtau$.
		\end{rmk}
		
		\subsubsection{Reduction of pressure and divergence}\label{sect reduction q}
		We show how to reduce the control of the pressure to that of the velocity and magnetic field when there is at least one spatial derivative on the fluid pressure $q$. This follows from using the momentum equation
		\begin{align}\label{mom 3}
			-\nab q=\eps \rho D_t u + \eps B\times(\nab\times B)
		\end{align}and the continuity equation and the divergence constraint
		\begin{align}\label{div 3}
			\nab\cdot u =-\eps a D_t q,~~\nab\cdot B=0.
		\end{align}
		These two equations show that  a \textit{spatial} derivative of $q$ is reduced to a \textit{material} derivative of $u$ and the curl part of $B$ and the divergence of velocity is converted to the material derivative of $q$.
		
		For the differentiated version, we aim to prove the following reduction
		\begin{lem}\label{lem div}
		Let $k,l$ be integers satisfying $0\leq l\leq 3$ and $0\leq k\leq 3-l$. Then we have
		\begin{align}
		\label{nabq estimate} \|(\eps D_t)^{k+2l}q\|_{4-k-l}^2\lesssim&~ \|\rho (\eps D_t)^{k+2l+1} u\|_{3-k-l}^2 + \|\eps B\times (\eps D_t)^{k+2l}(\nab\times B) \|_{3-k-l}^2\\
		 &+P(E(0)) + P(E(t))\int_0^t P(E(\tau))\dtau, \notag\\
		\label{divu estimate}  \|\nab\cdot (\eps D_t)^{k+2l}u\|_{3-k-l}^2\lesssim&~\|a(\eps D_t)^{k+2l+1} q\|_{3-k-l}^2 + P(E(0)) + P(E(t))\int_0^t P(E(\tau))\dtau,\\
		 \label{divB estimate}  \|\nab\cdot (\eps D_t)^{k+2l}B\|_{3-k-l}^2\lesssim&~ P(E(0)) + P(E(t))\int_0^t P(E(\tau))\dtau.
		\end{align}
		\end{lem} 
		\begin{proof}
		We start with the reduction of $\nab q$.
		\begin{align}\label{nab q}
		\|(\eps D_t)^{k+2l}q\|_{4-k-l}^2\lesssim \|(\eps D_t)^{k+2l}q\|_{0}^2 +\|(\eps D_t)^{k+2l}\nab q\|_{3-k-l}^2 +\|[(\eps D_t)^{k+2l},\nab]q\|_{3-k-l}^2,
		\end{align}where the first term on the right side has been controlled in Proposition \ref{Lem_TangentialEstimates}. This inequality together with \eqref{divcurl E4u}-\eqref{divcurl E5B} shows that we shall control the following terms
		\begin{itemize}
		\item $\|(\eps D_t)^{k+2l}\nab q\|_{3-k-l}^2$ and $\|(\eps D_t)^{k+2l}\nab \cdot u\|_{3-k-l}^2$.
		\item $\|[(\eps D_t)^{k+2l},\nab]f\|_{3-k-l}^2$ for $f=u,B,q$. 
		\end{itemize}
		For the commutator part, it is summarized in Lemma \ref{lem prod comm} and we just record the result here
		\begin{align}
			\left\|[\p, (\eps D_t)^{k+2l}]f\right\|_{3-k-l}^2\leq  P(E(0)) + P(E(t))\int_0^t P(E(\tau))\dtau.
		\end{align}Next, we control $\|(\eps D_t)^{k+2l}\nab q\|_{3-k-l}^2$ and $\|(\eps D_t)^{k+2l}\nab \cdot u\|_{3-k-l}^2$. Invoking the equations \eqref{mom 3}-\eqref{div 3}, we get
		\begin{align}
		 \label{Dt div u}\|(\eps D_t)^{k+2l}\nab \cdot u\|_{3-k-l}^2=\|(\eps D_t)^{k+2l}(a\eps D_t q)\|_{3-k-l}^2\leq \|a(\eps D_t)^{k+2l+1} q\|_{3-k-l}^2 + \|[(\eps D_t)^{k+2l},a]\eps D_t q\|_{3-k-l}^2,
		\end{align}where the last term can be directly controlled and we postpone the analysis to Section \ref{sect commutators}. For the pressure gradient, we have
		\begin{align*}
		-(\eps D_t)^{k+2l}\nab q=&~(\eps D_t)^{k+2l}(\rho\eps D_t u +\eps B\times(\nab\times B))\\
		=&~ \rho (\eps D_t)^{k+2l+1} u + \eps B\times (\eps D_t)^{k+2l}(\nab\times B) \\
		&+[(\eps D_t)^{k+2l}, \rho](\eps D_t u) + [(\eps D_t)^{k+2l}, \eps B\times](\nab\times B),
		\end{align*}and so
		\begin{align}
		 \label{Dt nab q}\|(\eps D_t)^{k+2l}\nab q\|_{3-k-l}^2\leq&~ \|\rho (\eps D_t)^{k+2l+1} u\|_{3-k-l}^2 + \|\eps B\times (\eps D_t)^{k+2l}(\nab\times B) \|_{3-k-l}^2\\
		&+\|[(\eps D_t)^{k+2l}, \rho](\eps D_t u)\|_{3-k-l}^2 + \|  [(\eps D_t)^{k+2l}, \eps B\times](\nab\times B)\|_{3-k-l}^2\notag
		\end{align}
		In Section \ref{sect commutators}, we will prove that these commutators terms can be directly controlled
		\begin{align}
		 &\|[(\eps D_t)^{k+2l},a]\eps D_t q\|_{3-k-l}^2+\|[(\eps D_t)^{k+2l}, \rho](\eps D_t u)\|_{3-k-l}^2 + \|  [(\eps D_t)^{k+2l}, \eps B\times](\nab\times B)\|_{3-k-l}^2\notag\\
		 \label{div comm} \lesssim&~ P(E(0)) + P(E(t))\int_0^t P(E(\tau))\dtau.
		\end{align}
		 Therefore, we conclude that the reduction of pressure and divergence actually converts a spatial derivative to a material derivative $(\eps D_t)$ and lower-order curl estimate that has been analyzed in Section \ref{sect curl}. We get
		 \begin{align}
		 \|(\eps D_t)^{k+2l}q\|_{4-k-l}^2\lesssim&~ \|\rho (\eps D_t)^{k+2l+1} u\|_{3-k-l}^2 + \|\eps B\times (\eps D_t)^{k+2l}(\nab\times B) \|_{3-k-l}^2\\
		 &+ P(E(0)) + P(E(t))\int_0^t P(E(\tau))\dtau, \notag\\
		 \|\nab\cdot (\eps D_t)^{k+2l}u\|_{3-k-l}^2\lesssim&~\|a(\eps D_t)^{k+2l+1} q\|_{3-k-l}^2 + P(E(0)) + P(E(t))\int_0^t P(E(\tau))\dtau, \notag\\
		\|\nab\cdot (\eps D_t)^{k+2l}B\|_{3-k-l}^2\lesssim&~ P(E(0)) + P(E(t))\int_0^t P(E(\tau))\dtau.
		 \end{align} The proof is completed.
		 \end{proof}
		\begin{rmk}
		 Lemma \ref{lem div} shows that we can reduce a spatial derivative on $q$ to $(\eps D_t u)$ and $\nab\times B$ (which is analyzed in Lemma \ref{lem curl}) plus controllable terms; and reduce the divergence part to $(\eps D_t q)$ plus controllable terms. Therefore, every time we do such reduction, we actually trade one spatial derivative for one material derivative together with an $\eps$ weight.  Repeatedly, we end this reduction procedure until there is no spatial derivatives falling on $u, B, q$, in which case we finally needs to the estimates of full material derivatives which have been analyzed in Proposition \ref{Lem_TangentialEstimates}.
		\end{rmk}
		
		\subsubsection{Estimates of commutators}\label{sect commutators}
		The end of this part is devoted to the control of commutator terms arising from \eqref{divcurl E4u}-\eqref{divcurl E5B}. We record the results here and refer to the proofs in Lemma \ref{lem prod comm} and \ref{lem Dt comm}.
		\begin{lem}\label{lem commutators}
		Let $k,l$ be integers satisfying $0\leq l\leq 3$ and $0\leq k\leq 3-l$. Then the following estimates hold
		\begin{align}
		&\left\|[(\eps D_t)^{k+2l},\nab\cdot](u,B)\right\|_{3-k-l}^2+\left\|[(\eps D_t)^{k+2l},\nab\times](u,B)\right\|_{3-k-l}^2\lesssim P(E(0)) + P(E(t))\int_0^t P(E(\tau))\dtau\\
		&\left\|[(\eps D_t)^{k+2l},a]\eps D_t q\right\|_{3-k-l}^2+\left\|[(\eps D_t)^{k+2l}, \rho](\eps D_t u)\right\|_{3-k-l}^2 + \left\|  [(\eps D_t)^{k+2l}, \eps B\times](\nab\times B)\right\|_{3-k-l}^2\\
		\lesssim&~ P(E(0)) + P(E(t))\int_0^t P(E(\tau))\dtau. \notag
		\end{align}
		\end{lem}
			
		\subsection{Closing the uniform estimates}
		Now we can turn to prove Theorem \ref{main thm, ill data}, that is, the uniform-in-$\eps$ control of $$E(t):=\sum_{l=0}^4\sum_{k=0}^{4-l}\left\|(\eps D_t)^{k+2l}\left(u,B,q,S,\rbp S\right)\right\|_{4-k-l}^2.$$ 
		\paragraph*{Control of the entropy.} The control of entropy and its directional derivative $\rbp S$ are established in Corollary \ref{Cor_Entropy}, that is,
		\begin{align}
		\ddt\left\|(\eps D_t)^{k+2l}(S,\rbp S)\right\|_{4-k-l}^2\leq P(E(t)).
		\end{align}
		
		\paragraph*{Control of the velocity $u$, the magnetic field $B$ and the pressure $q$.} This part is divided into several steps.
		\begin{enumerate}
		\item When $k+l=4$, there is no normal derivative in $\|\cdot\|_{4-k-l}$ norm. The control of this part is recorded in Proposition \ref{Lem_TangentialEstimates} by using the technique of modified Alinhac good unknowns.
		\item When $k+l\leq 3$, recall that \eqref{divcurl E4u}-\eqref{divcurl E5B} show that the norms $\|(\eps D_t)^{k+2l}(u,B)\|_{4-k-l}^2$ are reduced to the control of $\|(\eps D_t)^{k+2l}(\nab\times u,\nab\times B)\|_{3-k-l}^2$ and $\|(\eps D_t)^{k+2l}(\nab\cdot u)\|_{3-k-l}^2$. The commutators generated in this process can be directly controlled as shown in Lemma \ref{lem commutators}.
		\item For the curl part, we first show in Lemma \ref{lem curl 0} that 
		\begin{align}
		&\ddt\left(\left\|(\eps D_t)^{k+2l}\nab\times(\rho_0 u)\right\|_{3-k-l}^2+\left\|(\eps D_t)^{k+2l}\nab\times(\rho_0 B)\right\|_{3-k-l}^2+\left\|\eps B\times\left((\eps D_t)^{k+2l}\nab\times(\rho_0 B)\right)\right\|_{3-k-l}^2\right)\leq P(E(t)).
		\end{align}Then we remove the coefficient $\rho_0$ to reproduce the desired curl estimates
		\begin{align}
		&\left\|(\eps D_t)^{k+2l}\nab\times u\right\|_{3-k-l}^2 + \left\|(\eps D_t)^{k+2l}\nab\times B\right\|_{3-k-l}^2 \notag\\
		 \lesssim&~P\left(E(0)\right) + P(E(t)) \int_0^tP(E(\tau))\dtau +\sum_{k_0\leq k\atop l_0\leq l} \left(\|(\eps D_t)^{k_0+2l_0}u\|_{3-k_0-l_0}^2+\|(\eps D_t)^{k_0+2l_0}B\|_{3-k_0-l_0}^2\right),
		\end{align}	 which are recorded in Lemma \ref{lem curl}. 
		\item For the divergence part, we combine the momentum equation and the continuity equation to reduce $\nab q$ to $\eps D_t u$ and $\eps \nab\times B$ and reduce $\nab\cdot u$ to $\eps D_t q$:
		 \begin{align}
		 \|(\eps D_t)^{k+2l}q\|_{4-k-l}^2\lesssim&~ \|\rho (\eps D_t)^{k+2l+1} u\|_{3-k-l}^2 + \|\eps B\times (\eps D_t)^{k+2l}(\nab\times B) \|_{3-k-l}^2\\
		 &+ P(E(0)) + P(E(t))\int_0^t P(E(\tau))\dtau, \notag\\
		 \|\nab\cdot (\eps D_t)^{k+2l}u\|_{3-k-l}^2\lesssim&~\|a(\eps D_t)^{k+2l+1} q\|_{3-k-l}^2 + P(E(0)) + P(E(t))\int_0^t P(E(\tau))\dtau, \notag\\
		\|\nab\cdot (\eps D_t)^{k+2l}B\|_{3-k-l}^2\lesssim&~ P(E(0)) + P(E(t))\int_0^t P(E(\tau))\dtau.
		 \end{align}as recorded in Lemma \ref{lem div}.
		 \item It should be noted that the extra terms generated on the right side of curl estimates are of lower order. If $k_0+l_0=3$, then again we can control such terms by the tangential estimates in Proposition \ref{Lem_TangentialEstimates}. If $k_0+l_0<3$, then we again apply the above div-curl analysis to such terms to reduce one more spatial derivative.
		 \item We repeat step 3 and 4 until there is no spatial derivatives falling on $u, B, q$, and all remaining terms are $u,B,q$ differentiated by $(\eps D_t)^j$ for some $j$ which are again controlled in Proposition \ref{Lem_TangentialEstimates}:
		 \begin{align}
				\sum_{j=1}^{8} \sup_{\tau\in[0,t]}\left\|(\eps D_t)^j(q,u,B)\right\|_0^2 \le	P(E(0)) +  E(t)\int_0^t P(E(\tau)) \mathrm{d}\tau.	
			\end{align}
		\end{enumerate}
		The above argument gives the following uniform-in-$\eps$ estimates:
		\begin{equation}
			E(t)\lesssim P(E(0)) + P(E(t))\int_0^t P(E(\tau))\dtau
		\end{equation}
		where $E(t)$ is defined by \eqref{energy intro}. Since the right side of the energy inequality does not rely on $\eps$, we can use Gr\"onwall-type argument to prove that there exists some $T>0$ independent of $\eps$ such that
		\begin{equation}
			\sup_{t\in[0,T]} E(t) \le P(E(0)).
		\end{equation}
		Theorem \ref{main thm, ill data} is proven.

		\section{The limit to the incompressible inhomogeneous MHD system}\label{sect limit}
		
		The last section is devoted to the proof of {\it strong convergence} to the incompressible inhomogeneous MHD system \eqref{IMHD}. 
		
		\subsection{The strong convergence of vorticity, magnetic field and entropy}
		The uniform bounds imply that, up to a subsequence, we have
		\begin{align}
			(q,u,B,S) \to (q^0,u^0,B^0,S^0) &\text{ weakly-* in } L^{\infty}([0,T];H^4(\Omega)),\\
			(B,S)\to (B^0,S^0) &\text{ strongly in } C([0,T];H^{4-\delta}_{\mathrm{loc}}(\Omega)),\label{StrongLim_S_B}\\
			\nabla\times(\rho_0(S)u) \to \nabla\times(\rho_0(S^0)u^0) &\text{ strongly in } C([0,T];H^{3-\delta}_{\mathrm{loc}}(\Omega)),\label{StrongLim_curlrho0u}
		\end{align}
		with $\rho_0(S)=\rho(0,S)$. Similarly, we use $a_0(S)$ to denote $a(0,S)$. Here we note that the strong convergence of the vorticity is due to the uniform boundness of $D_t\nab\times(\rho_0 u)$ which satisfies
		\[
		D_t( \nabla\times(\rho_0 u) )=\nabla\times\left( \rho_0\rho^{-1} B\cdot\nabla B \right) + [u\cdot\nab,\nabla\times](\rho_0 u) + \nabla g \times\nabla (q+ \eps|B|^2/2),
		\]where $g(\eps q, S)=(\eps\rho)^{-1}(\rho-\rho_0)$ is bounded uniformly in $\eps$.
		
		\subsection{The strong convergence of pressure and divergence}
		We first prove that the expected limit functions satisfy $q^0=0$ and $\nabla\cdot u^0=0$. The first and second equation of \eqref{CMHD3} are written to be 
		\begin{equation}
			E(\eps q, S) D_t U + \eps^{-1} L U= J,
		\end{equation}
		with
		\begin{equation*}
			E(\eps q, S)=\begin{pmatrix}
				a(\eps q, S) & 0 \\
				0 & \rho(\eps q, S) I_3
			\end{pmatrix},\quad L=\begin{pmatrix}
				0 &\nabla\cdot\\
				\nabla & 0
			\end{pmatrix},\quad U=\begin{pmatrix}
				q \\
				u
			\end{pmatrix},\quad J=\begin{pmatrix}
				0\\
				B\times(\nabla\times B)
			\end{pmatrix}.
		\end{equation*}
		First notice that
		\begin{equation}
			\eps E(\eps q, S) \p_t U + L U = \eps E(\eps q, S) u\cdot\nabla U -\eps J.
		\end{equation}
		Using the uniform bounds and $E(\eps q,S)-E_0(S)=O(\eps)$, we obtain
		\begin{equation}\label{EuqU}
			\eps E_0(S)\p_t U + L U= \eps f,
		\end{equation} 
		where $E_0(S)= E(0,S)$ and $\{f\}_{\eps>0}$ is a bounded family in $C([0,T];H^3(\Omega))$. Passing to the weak limit shows that $\nabla q^0=0$ and $\nabla\cdot u^0$=0. Since $q^0\in L^\infty([0,T];H^4(\Omega))$ and $\Omega=\mathbb{R}^d_+$, we infer $q^0=0$.

		\begin{prop}\label{StrongLim_q_divu}
			It holds that
			\begin{align}
				q\to& 0 \text{ strongly in } L^2([0,T];H^{4-\delta}_{\mathrm{loc}}(\Omega)),\\
				\nabla\cdot u \to & 0 \text{ strongly in } L^2([0,T];H^{3-\delta}_{\mathrm{loc}}(\Omega)).\label{StrongLim_divu}
			\end{align}
		\end{prop}
			This theorem is a slight variant of \cite[Prop. 3.1]{Alazard2005limit}. For the reader's convenience, we outline the proof detailed in \cite{Alazard2005limit} and skip some technical details that are identical to \cite[Prop. 3.1]{Alazard2005limit}. 
		\begin{proof}
		\textbf{Step 1: Wave-packet transform of the variable-coefficient system.}
			One first extends the functions to $t\in \mathbb{R}$ by
			\begin{equation}
				\tilde{U}=\begin{pmatrix}
					\tilde{q}\\
					\tilde{u}
				\end{pmatrix}=\chi_{\eps} U =\begin{pmatrix}
					\chi_{\eps} q\\
					\chi_{\eps}u
				\end{pmatrix},
			\end{equation}
			where $\chi_{\eps}\in C_0^\infty((0,T))$ be a family of functions such that $\chi_{\eps}(t)=1$ for $t\in [\eps^{1/2},T-\eps^{1/2}]$ and $\|\eps\p_t \chi_{\eps}\|_{\infty} \le 2\eps^{1/2}$, and choose extensions $\tilde{S}$ of $S$, supported in $t\in [-1,T+1]$, uniformly bounded in $C(\mathbb{R};H^4(\Omega))$, and converging to $\tilde{S}^0$ in $C(\mathbb{R};H^{4-\delta}_{\mathrm{loc}}(\Omega))$. According to \eqref{EuqU}, $\tilde{U}$ satisfies
			\begin{equation}\label{EuqTildeU}
				\eps E_0(\tilde{S})\p_t \tilde{U} + L \tilde{U}= \eps \tilde{f},
			\end{equation} 
			where $\{\tilde{f}\}_{\eps>0}$ is a bounded family in $C(\mathbb{R};H^3(\Omega))$. 
			
			The above extension to $t\in\R$ is necessary, as we want to apply the wave-packet transform to the variable $t$. Given $\eps>0$, we define the following wave packet transform:
			\begin{equation}
				W^{\eps} v(t,\tau,x)= (2\pi^3)^{-1/4}\eps^{-3/4} \int_{\mathbb{R}} e^{(\mathrm{i} (t-s)\tau -(t-s)^2 )/\eps } v(s,x) \mathrm{d} s,
			\end{equation}
			where $v\in C^1(\mathbb{R}\times\bar{\Om}) \cap L^2(\mathbb{R}\times \Omega)$,  $W^{\eps} v\in C^1(\mathbb{R}^2_{t,\tau}\times\bar{\Om}) \cap L^2(\mathbb{R}^2_{t,\tau}\times \Omega)$ and $W^{\eps}$ extends as an isometry from $L^2(\mathbb{R}\times\Omega)$ to $L^2(\mathbb{R}^2_{t,\tau}\times\Omega)$.
			
			Now, system \eqref{EuqTildeU} can be written in $\mathbb{R}^2_{t,\tau}\times \Omega$ as
			\begin{equation}\label{EquWTildeU}
				\mathrm{i}\tau E_0(\tilde{S})( W^{\eps}\tilde{U}) + L ( W^{\eps}\tilde{U} ) = \boldsymbol{F}^{\eps},
			\end{equation}
			where 
			\begin{align*}
				\boldsymbol{F}^{\eps}=&~\eps W^{\eps} \tilde{f} + [ E_0(\tilde{S}),W^{\eps} ](\eps\p_t) \tilde{U} + E_0(\tilde{S}) ( \mathrm{i}\tau W^{\eps} \tilde{U}- W^{\eps} (\eps\p_t \tilde{U} )  ) \\
				:=&~( F^{\eps}_1, \boldsymbol{F}_2^{\eps} ) \in L^2(\mathbb{R}^2_{t,\tau};H^1(\Omega)  ) \times L^2(\mathbb{R}^2_{t,\tau};(H^1(\Omega))^d ).
			\end{align*}
			Following the arguments in \cite[Lemma 3.3]{Alazard2005limit}, one shows that
			\begin{equation}\label{FepsTo0}
				\boldsymbol{F}^{\varepsilon}\to 0 \text{ in } L^2(\mathbb{R}^2_{t,\tau};H^1(\Omega)) \text{ as } \varepsilon\to 0.
			\end{equation}
			
			\textbf{Step 2: Strong convergence of $q$ via the techniques of microlocal defect measures.} 
			Now we turn to prove the convergence of $q$. Since $W^\eps$ is an isometry, it suffices to prove the strong convergence of $W^\eps\tilde{q}$ in certain function spaces on $\R_{t,\tau}^2\times\Om$. We define
			\begin{align}
				P^{\eps}(t,\tau,\nabla) (\cdot):=&~ a_0(\tilde{S})\tau^2 (\cdot) + \nabla\cdot ( \rho_0^{-1}(\tilde{S}) \nabla(\cdot) ),\\
				P^{0}(t,\tau,\nabla) (\cdot):=&~ a_0(\tilde{S}^0)\tau^2 (\cdot) + \nabla\cdot ( \rho_0^{-1}(\tilde{S}^0) \nabla(\cdot) ),\\
				\Theta^{\eps}:=&~ (1 - \Delta) (W^{\eps} \tilde{q}) \in L^2(\mathbb{R}^2_{t,\tau}\times \Omega) .
			\end{align} It should be noted that $P^\eps$ is actually the wave-packet transform of the wave operator $a_0\eps^2\p_t^2 (\cdot) + \nab\cdot(\rho_0^{-1}(S)\nab(\cdot))$. From \eqref{EquWTildeU}, we can compute that
			\begin{equation}
				P^{\eps}(t,\tau,\nabla) (W^{\eps}\tilde{q})= -\mathrm{i}\tau F_1^{\eps} + \nabla\cdot( \rho_0^{-1}(\tilde{S}) \boldsymbol{F}_2^{\eps}  ).
			\end{equation} Since this is a boundary-value problem, we shall decompose $W^\eps \tilde{q}$ into its interior part and boundary part. Following \cite{Alazard2005limit}, we define
			 \begin{align}\label{decompose Wq}
				W^\eps \tilde{q} = (1-\lap_N)^{-1}\Theta^\eps + \mathfrak{N} (\boldsymbol{F}_2^\eps\cdot N),
			\end{align}where $\Theta:=(1-\lap)(W^\eps\tilde{q})$ and $(1-\Delta_N)^{-1}$ is defined by
			\begin{equation*}
				f=(1-\Delta_N)^{-1}g \text{ if and only if } (1-\Delta)f=g \text{ in }\Omega,
				\text{ and } \p_N f=0 \text{ on }\Sigma;	
			\end{equation*}and $\mathfrak{N}$ is defined by
			\begin{equation*}
				h=\mathfrak{N}(g) \text{ if and only if } (1-\Delta)h=0 \text{ in }\Omega,
				\text{ and } \p_N h=g \text{ on }\Sigma.
			\end{equation*} It should be noted that $(1-\lap_N)^{-1}$ is a bounded linear operator from $L^2(\Om)$ to $H^2(\Om)$ and $\mathfrak{N}$ is a bounded linear operator from $H^{\frac12}(\Sigma)$ to $H^2(\Om)$.
			
			To prove the strong convergence of $\Theta$, which is now only a uniformly bounded family in $L^2(\R^2\times\Om)$, we need the following two lemmas.
			\begin{lem}[M\'etivier-Schochet {\cite[Lemma 4.3]{Metivier2001limit}}]\label{lem MS}
				For all uniformly bounded family $\{\Theta^\eps\}\subset L^2(\R^{2+d})$, there is a subsequence such that there exists a finite non-negative Borel measure $\mu$ on $\R^2$ and $M\in L^1(\R^2,\mathcal{L}_+,\mu)$ such that for all $\Phi\in C_0(\R^2;\mathcal{K})$,
				\[
				\int_{\R^2}\left( \Phi \Theta^\eps, \Theta^\eps \right)_{L^2}\dt\dtau \xrightarrow{\eps \to 0}\int_{\R^2} \text{Tr}(\Phi(t,\tau)M(t,\tau))\mu(\dt,\dtau).
				\]Here $\mathcal{K}$ ($\mathcal{L}_+$, resp.) denotes the set of compact operators (non-negative self-adjoint trace class operators, resp.) on $L^2(\Om)$.
			\end{lem}
			
			\begin{lem}[M\'etivier-Schochet {\cite[Lemma 5.1]{Metivier2001limit}}]\label{lem ker}
				The operator $P^0(t,\tau,\nab)(1-\lap_N)^{-1}=0$ is a 1-1 mapping for any $(t,\tau)\in\R^2$, that is,
				\begin{align}
					\ker_{L^2(\Omega)} ( P^0(t,\tau,\nabla)(1-\Delta_N)^{-1}  ) =0, \quad \forall (t,\tau)\in \mathbb{R}^2\label{KerP0}
				\end{align}
			\end{lem}
			\begin{rmk}
				We note that the first lemma give the description of ``the lacking of compactness in $L^2$". For the second lemma, its proof (see \cite[Lemma 5.1]{Metivier2001limit}) requires the entropy decay condition \eqref{EntropyDecay} and the unboundedness of the domain $\Om$. 
			\end{rmk}
			Let $M(t,\tau)$ be the trace-class operator and $\mu$ be the microlocal defect measure obtained in Lemma \ref{lem MS} by inserting $\Theta^\eps$ defined in \eqref{decompose Wq}. Then we can prove 
			\begin{align}
					M(t,\tau)=0 \quad \mu\text{-a.e.}, \label{P0Ma.e.}
			\end{align} whose details can be found in \cite[Corollary 4.4]{WZ2024elasto}. Therefore, for the uniformly bounded family $\{\Theta^\eps\}$ defined in \eqref{decompose Wq}, we actually prove that
			
			\[
				\int_{\R^2}\left( \Phi \Theta^\eps, \Theta^\eps \right)_{L^2}\dt\dtau \xrightarrow{\eps \to 0}0
			\] holds for any $\Phi\in C_0(\R^2;\mathcal{K})$ where $\mathcal{K}$ denotes the set of compact operators on $L^2(\Om)$.

			We now set $\Phi(t,\tau)=\varphi(t,\tau)K^*K$ for $\varphi\in C_0(\R^2)$ and $K\in\mathcal{K}$ in the above convergence result to get
			\begin{equation}
				\varphi K\Theta^{\eps} \to 0 \text{ in } L^2(\mathbb{R}^2_{t,\tau}\times\Omega)
			\end{equation}holds for any $K\in \mathcal{K}$ and $\varphi\in C_0(\R^2)$.
			Following the arguments in \cite[(3.23)-(3.24)]{Alazard2005limit}, we prove this convergence holds for $\varphi(t,\tau)=1$, i.e. for any $K\in \mathcal{K}$,
			\begin{equation}\label{ThetaCovergence1}
				K \Theta^{\eps} \to 0 \text{ in } L^2(\mathbb{R}^2_{t,\tau}\times\Omega).
			\end{equation}
			Recall that, by the definition of $\Theta=(1-\Delta) (W^{\eps} \tilde{q})$, $W^{\eps}$ is an isometry from $L^2(\mathbb{R}_t\times\Omega)$ to $L^2(\mathbb{R}^2_{t,\tau}\times \Omega)$, and $W^{\eps}$ commutes with $K(1-\Delta)$. So \eqref{ThetaCovergence1} implies that for any $K\in \mathcal{K}$,
			\begin{equation}\label{qConvergence1}
				K(1-\Delta)\tilde{q}\to 0 \text{ in } L^2(\mathbb{R}\times\Omega).
			\end{equation}
			Given that $\tilde{q}$ is bounded in $L^2(\mathbb{R};H^4(\Omega))$, the convergence \eqref{qConvergence1} implies
			\begin{equation}
				\tilde{q}\to 0 \text{ in } L^2(\mathbb{R};H^{4-\delta}_{\mathrm{loc}}(\Omega)).
			\end{equation}
			Since the limit is $0$, the convergence holds without passing a subsequence. We end up with
			\begin{equation}
				q\to 0 \text{ in } L^2([0,T];H^{4-\delta}_{\mathrm{loc}}(\Omega)).
			\end{equation}
			Similarly, from \eqref{EquWTildeU}, we can prove
				\begin{equation}
					\eps\p_t q\to 0 \text{ in } L^2([0,T];H^{3-\delta}_{\mathrm{loc}}(\Omega)),
				\end{equation}
			which leads to
			\begin{equation}
				\nabla \cdot u = -a D_t q\to 0 \text{ in } L^2([0,T];H^{3-\delta}_{\mathrm{loc}}(\Omega)).
			\end{equation}
			
		\end{proof}
		
		\subsection{The strong convergence to the limit system}
		We continue our proof of Theorem \ref{main thm, limit}. Recall that $\mathcal{P}$ be the projection onto $H_{\sigma}$ and $\mathcal{Q}= I_3 - \mathcal{P}$, where $H_\sigma=\{u\in L^2(\Omega): \int_{\Om} u\cdot\nabla\phi,\ \forall \phi\in H^1(\Omega)\}$ and $G_{\sigma}=\{\nabla\psi: \psi\in H^1(\Omega)\}$ give the orthogonal decomposition $L^2(\Omega)=H_{\sigma} \oplus G_{\sigma}$. 
		
		From the strong convergence results \eqref{StrongLim_curlrho0u} and \eqref{StrongLim_divu}, we know that
		\begin{align}
			\mathcal{P}( \rho_0(S)u )\to \mathcal{P}( \rho_0(S^0)u^0 ) &\text{ in } L^2([0,T];H^{4-\delta}_{\mathrm{loc}}(\Omega)),\\
			\mathcal{Q}u\to \mathcal{Q}u^0=0 &\text{ in } L^2([0,T];H^{4-\delta}_{\mathrm{loc}}(\Omega)).\label{StrongLim_Qu}
		\end{align}
		The previous two properties yields further that:
		\begin{align}
			\mathcal{P}( \rho_0(S)\mathcal{P}u ) \to \mathcal{P}(\rho_0(S^0)\mathcal{P}u^0) &\text{ in } L^2([0,T];H^{4-\delta}_{\mathrm{loc}}(\Omega)),\\
			\mathcal{P}( \rho_0(S)\mathcal{Q}u ) \to \mathcal{P}(\rho_0(S^0)\mathcal{Q}u^0)=0 &\text{ in } L^2([0,T];H^{4-\delta}_{\mathrm{loc}}(\Omega)),
		\end{align}
		which, combined with the fact $S\to S^0$ in $C([0,T];H^{4-\delta}_{\mathrm{loc}}(\Omega))$, imply that:
		\begin{align*}
			\mathcal{P}( \rho_0(S^0)\mathcal{P}(u-u^0) )=&~ \mathcal{P}( \rho_0(S^0)[ (u-u^0)- \mathcal{Q}u ]  ) \\
			=&~\mathcal{P}\left(  \rho_0(S)(u-u^0) + ( \rho_0(S^0)-\rho_0(S) )(u-u^0) - \rho_0(S) \mathcal{Q} u + (\rho_0(S^0)-\rho_0(S))\mathcal{Q} u  \right)\\
			\to &~0 \text{ in } L^2([0,T];H^{4-\delta}_{\mathrm{loc}}(\Omega)).
		\end{align*}
		Consequently, we find that, by noticing $\rho_0(S^0)$ positive in $[0,T]\times \Omega$, $\mathcal{P}u\to \mathcal{P}u^0= u^0$ in $L^2([0,T];H^{4-\delta}_{\mathrm{loc}}(\Omega))$, which, combined with \eqref{StrongLim_Qu}, implies that
		\begin{equation}\label{StrongLim_u}
			u\to u^0 \text{ in } L^2([0,T];H^{4-\delta}_{\mathrm{loc}}(\Omega)).
		\end{equation}
		By \eqref{StrongLim_S_B} and \eqref{StrongLim_u}, we obtain
		\begin{align*}
			\rho(\eps q,S)\to \rho_0(S^0) &\text{ in } C([0,T];H^{4-\delta}_{\mathrm{loc}}(\Omega)),\\
			\nabla u\to \nabla u^0 &\text{ in } L^2([0,T];H^{3-\delta}_{\mathrm{loc}}(\Omega)),\\
			\nabla B\to \nabla B^0 &\text{ in } L^2([0,T];H^{3-\delta}_{\mathrm{loc}}(\Omega)).
		\end{align*}
		Passing to the limit in the equations for $S$ and $B$, we see that the limits $S^0$ and $B^0$ satisfy
		\begin{equation*}
			(\partial_t +u^0\cdot\nabla) S^0=0,\quad (\partial_t + u^0\cdot\nabla)B^0 = B^0 \cdot\nabla u^0,\quad \nabla\cdot B^0=0
		\end{equation*}
		in the sense of distributions. Since $\rho(\eps q, S)- \rho_0(S)=O(\eps)$, we have
		\begin{align*}
			\rho(\eps q, S) D_t u=&~ (\rho(\eps q, S)- \rho_0(S))D_t u + \p_t ( \rho_0(S)u ) + (u\cdot \nabla)( \rho_0(S) u )\\
			\to &~\rho_0(S^0) ( (\p_t +u^0\cdot\nabla)u^0  )
		\end{align*}
		in the sense of distributions. Applying the operator $\mathcal{P}$ to the momentum equations $\rho D_t u + \eps^{-1}\nabla q +B\times(\nabla\times B)=0$ and then taking to the limit, we conclude that
		\begin{equation*}
			\mathcal{P}\left[  \rho_0(S^0)( (\p_t +u^0\cdot\nabla)u^0 + B^0\times(\nabla\times B^0) )  \right]=0.
		\end{equation*}
		Therefore, $(u^0,B^0,S^0)\in C([0,T];H^4(\Om))$ solves the incompressible MHD equations together with a transport equation
		\begin{equation}
			\begin{cases}
				\varrho(\p_t u^0 + u^0\cdot\nab u^0) -B^0\cdot\nab B^0+ \nab (\pi+\frac12|B^0|^2) =0&~~~ \text{in}~[0,T]\times \Omega,\\
				\p_t B^0+ u^0\cdot\nab B^0 -B^0\cdot\nab u^0=0 &~~~ \text{in}~[0,T]\times \Omega,\\
				\nab\cdot u^0=\nab\cdot B^0=0&~~~ \text{in}~[0,T]\times \Omega,\\
				\p_t S^0+u^0\cdot\nab S^0=0&~~~ \text{in}~[0,T]\times \Omega,\\
				u_d^0=B_d^0=0&~~~\text{on}~[0,T]\times\Sigma,
			\end{cases}
		\end{equation}
		for a suitable fluid pressure function $\pi$ satisfying $\nab\pi\in C([0,T];H^3(\Om))$. Here $\varrho$ satisfies $\p_t\varrho+u^0\cdot\nab\varrho=0,$ with initial data $\varrho_0:=\rho(0,S_0^0)$. Using the same arguments as in the proof of \cite[Theorem 1.5]{Metivier2001limit}, we find that 
		\begin{equation*}
			(u^0,B^0,S^0)|_{t=0}=(w_0, B^0_0,S^0_0),
		\end{equation*}
		where $w_0 \in H^4(\Omega)$ is determined by
		\begin{equation*}
			w_{0d}|_{\Sigma}=0,\quad \nabla\cdot w_0=0,\quad \nabla\times( \rho_0(S_0^0)w_0 )=\nabla\times( \rho_0(S_0^0)u_0^0 ).
		\end{equation*}
		Moreover, the uniqueness of the limit function implies that the convergence holds as $\eps\to 0$ without restricting to a subsequence. Theorem \ref{main thm, limit} is then proven.

		\paragraph*{Acknowledgment.} 
		Qiangchang Ju is supported by the National Natural Science Foundation of China (Grants 12131007). Jiawei Wang is supported by the National Natural Science Foundation of China (Grants 12131007) and the Basic Science Center Program (No: 12288201) of the National Natural Science Foundation of China.
		
		\paragraph*{Data avaliability.} This manuscript has no associated data.

		\subsection*{Ethics Declarations}
		\paragraph*{Conflict of interest.} The authors declare that there is no conflict of interest.

		\begin{appendix}
			
			\section{Preliminary lemmas and commutator estimates}

			
			In this section, we recall some basic identities and estimates. First, we record the Hodge-type elliptic estimates that are used in the div-curl analysis.
			
			\begin{lem}[Hodge elliptic estimates]\label{hodgeTT}
				For any sufficiently smooth vector field $X\in\R^3$ and any real number $s\geq 1$, one has
				\begin{equation}
					\|X\|_s^2\lesssim \|X\|_0^2+\|\nab\cdot X\|_{s-1}^2+\|\nab\times X\|_{s-1}^2 + | X\cdot N|_{s-1/2}^2.
				\end{equation}  When $X\in\R^2$, we shall replace $\nab\times X$ by $\nab^\perp \cdot X$ with $\nab^\perp:=(-\p_2,\p_1)^\mathrm{T}.$
			\end{lem}
			
			\begin{lem}[Reynolds transport formula]\label{lem transport}
				Under the setting of Theorem \ref{main thm, ill data}, for each $f(t,x)$ satisfying $f\in L^2(\Om)$ and $D_t f\in L^2(\Om)$, the Reynolds transport formula holds
				\begin{equation}
					\frac12\ddt\io\rho |f|^2\dx=\io \rho (D_t f) f \dx.
				\end{equation} 
				Moreover, for $0\leq t\leq T$ with $T>0$ a given number, we have  
				\begin{align}
				\|f(t,\cdot)\|_0^2 \lesssim \|f(0,\cdot)\|_0^2 + \int_0^t \|D_t f(\tau,\cdot)\|_0^2\dtau
				\end{align}
			\end{lem}
			\begin{proof}
			The Reynolds transport formula is proved by direct calculation and using the continuity equation $D_t\rho +\rho\nab\cdot u = 0$. For the second formula, we use the non-degeneracy of $\rho$ ($1\lesssim\rho\lesssim 1$) to compute that
			\begin{align*}
			\|f(t,\cdot)\|_0^2 \lesssim&~ \io \rho |f(t,x)|^2\dx = \io \rho |f(0,x)|^2\dx + \int_0^t \io 2 \rho (D_t f(\tau,x)) f(\tau,x) \dx\dtau\\
			\leq&~\|f(0,\cdot)\|_0^2 + \int_0^t \|D_t f(\tau,\cdot)\|_0 \|f(\tau,\cdot)\|_0 \|\rho\|_\infty \dtau\\
			\lesssim&~\|f(0,\cdot)\|_0^2 + \int_0^t \|D_t f(\tau,\cdot)\|_0 \|f(\tau,\cdot)\|_0  \dtau
			\end{align*}
			Using Young's inequality, $ \|D_t f(\tau,\cdot)\|_0 \|f(\tau,\cdot)\|_0\leq \delta\|f(\tau,\cdot)\|_0^2 + (4\delta)^{-1}\|D_t f(\tau,\cdot)\|_0^2$. Therefore, we get for 
			\begin{align*}
			\sup_{0\leq t\leq T}\|f(t,\cdot)\|_0^2 \lesssim \delta \sup_{0\leq t\leq T}\|f(t,\cdot)\|_0^2  + \|f(0,\cdot)\|_0^2 + \int_0^t \|D_t f(\tau,\cdot)\|_0^2  \dtau.
			\end{align*}Pick $\delta>0$ suitably small such that the $\delta$-term is absorbed by the left side. Then we get
			\[
			\|f(t,\cdot)\|_0^2\leq \sup_{0\leq t\leq T}\|f(t,\cdot)\|_0^2 \lesssim  \|f(0,\cdot)\|_0^2 + \int_0^t \|D_t f(\tau,\cdot)\|_0^2  \dtau.
			\]
			\end{proof}
			
			The next lemma reveals the structure of the commutator between higher-order material derivatives and spatial derivatives.
			\begin{lem}[{\cite[Section 4]{Luo2018CWW}}]\label{Prop_AGU}
				
				Let $[\p, D_t]=(\p u)\tilde{\cdot}\p$, where the symmetric dot product $(\p u)\tilde{\cdot} \p$ is define component-wisely by $((\p u)\tilde{\cdot}\p)_i=\p_i u_k \p_k$. In general, we have
				\begin{align}
					[\p, D_t^k]=&~ \sum_{l_1+l_2=k-1} c_{l_1,l_2} ( \p D_t^{l_1}u )\tilde{\cdot}(\p D_t^{l_2}) + \sum_{l_1+\cdots+l_n=k-n+1 \atop n\ge 3} d_{l_1,\cdots,l_n}(\p D_t^{l_1}u)\cdots(\p D_t^{l_{n-1}}u)(\p D_t^{l_n})\notag\\
					=&~ (\p D_t^{k-1} u)\tilde{\cdot} \p + k (\p u)\tilde{\cdot}( \p D_t^{k-1} ) \notag\\
					&~+\sum_{l_1+l_2=k-1\atop l_1,l_2 >0 } c_{l_1,l_2} ( \p D_t^{l_1}u )\tilde{\cdot}(\p D_t^{l_2}) + \sum_{l_1+\cdots+l_n=k-n+1 \atop n\ge 3} d_{l_1,\cdots,l_n}(\p D_t^{l_1}u)\cdots(\p D_t^{l_{n-1}}u)(\p D_t^{l_n})\notag\\
					=&~ \p( D_t^{k-1} u_i\p_i (\cdot) ) + k \p u_i\p_i D_t^{k-1}(\cdot)\notag\\
					&~\underbrace{- D_t^{k-1} u_i\p_i\p + \sum_{l_1+l_2=k-1\atop l_1,l_2 >0 } c_{l_1,l_2} ( \p D_t^{l_1}u )\tilde{\cdot}(\p D_t^{l_2}) + \sum_{l_1+\cdots+l_n=k-n+1 \atop n\ge 3} d_{l_1,\cdots,l_n}(\p D_t^{l_1}u)\cdots(\p D_t^{l_{n-1}}u)(\p D_t^{l_n}) }_{:= Z^k},
				\end{align}
				where $Z^k=(Z_1^k,Z_2^k,Z_3^k)^\mathrm{T}$. Moreover, we have
				\begin{align}
					[\p, (\eps D_t)^k]=&~ \sum_{l_1+l_2=k-1} c_{l_1,l_2} (\eps \p (\eps D_t)^{l_1}u )\tilde{\cdot}(\p (\eps D_t)^{l_2}) + \sum_{l_1+\cdots+l_n=k-n+1 \atop n\ge 3} d_{l_1,\cdots,l_n}(\eps\p (\eps D_t)^{l_1}u)\cdots(\eps\p(\eps D_t)^{l_{n-1}}u)(\p (\eps D_t)^{l_n})\notag\\
					=&~ (\eps\p (\eps D_t)^{k-1} u)\tilde{\cdot} \p + k (\p u)\tilde{\cdot}(\eps \p (\eps D_t)^{k-1} ) \notag\\
					&~+\sum_{l_1+l_2=k-1\atop l_1,l_2 >0 } c_{l_1,l_2} ( \eps\p (\eps D_t)^{l_1}u )\tilde{\cdot}(\p (\eps D_t)^{l_2}) + \sum_{l_1+\cdots+l_n=k-n+1 \atop n\ge 3} d_{l_1,\cdots,l_n}(\eps\p (\eps D_t)^{l_1}u)\cdots(\eps\p (\eps D_t)^{l_{n-1}}u)(\p (\eps D_t)^{l_n}).
				\end{align} 
			\end{lem}
			
			\begin{lem}\label{lem prod comm}
			Given integers $k,l$ satisfying $0\leq l\leq 3$ and $0\leq k\leq 3-l$, under the setting of Theorem \ref{main thm, ill data}, we have that
			\begin{align}
			\left\|[\p, (\eps D_t)^{k+2l}]f\right\|_{3-k-l}^2\leq P(E(0)) + P(E(t))\int_0^t P(E(\tau))\dtau,
			\end{align}
			for $f=u,B,q,S$.
			\end{lem}
			\begin{proof}
			Here we only present the proof for the most difficult case: $l=3, k=0$ because the number of total terms in this commutator is greater than any others. Also, we assume $f=u$ without loss of generality. In view of Lemma \ref{Prop_AGU}, we can see that $[\p, (\eps D_t)^6]f$ contains the following terms with certain coefficients
			\begin{itemize}
			\item Quadratic: $(\eps \p (\eps D_t)^5 u)(\p u)$, $(\eps \p (\eps D_t)^4 u)(\p(\eps D_t) u)$, $(\eps \p (\eps D_t)^3 u)(\p(\eps D_t)^2 u)$.
			\item Cubic: $(\eps \p (\eps D_t)^4 u)(\eps \p u)(\p u)$, $(\eps \p (\eps D_t)^3 u)(\eps\p (\eps D_t)u)(\p u)$, $(\eps \p (\eps D_t)^2 u)(\eps\p (\eps D_t)^2 u)(\p u)$,\\ $(\eps \p (\eps D_t)^2 u)(\eps\p (\eps D_t)u)(\p (\eps D_t) u)$.
			\item Quartic and more low-order terms.
			\end{itemize}
			We notice that, in each monomial, there is exactly one term that does not have an extra $\eps$ weight, and we always set the lowest-order factor to be this term. Now we do the product estimates.
			
			The leading-order part is $(\eps \p (\eps D_t)^5 u)(\p u)$. Using Lemma \ref{lem transport}, we get
			\begin{align*}
			&\|(\eps \p (\eps D_t)^5 u)(\p u)\|_0^2 \leq \|\eps^6 D_t^5 u\|_1^2 \|\p u\|_\infty^2\\
			\lesssim&~\eps^2\|(\eps D_t)^5 u(0,\cdot)\|_1^2\|\p u\|_2^2 + \underbrace{\|\p u\|_2^2\int_0^t \|(\eps D_t)^6 u(\tau,\cdot)\|_1^2 \dtau}_{\leq E_4(t)\int_0^t P(E_7(\tau))\dtau},\\
			\lesssim&~\|\eps \p u\|_2^4 + \|(\eps D_t)^5 u(0,\cdot)\|_1^4 + E_4(t)\int_0^t P(E_7(\tau))\dtau
			\end{align*} Next, we apply Lemma \ref{lem transport} again to get
			\begin{align*}
			\|\eps \p u(t,\cdot)\|_2^4 \leq \|\eps u(t,\cdot)\|_3^4 \lesssim \left(\|\eps u_0\|_3^2 + \int_0^t \|\eps D_t u(\tau,\cdot)\|_3^2\dtau \right)^2
			\end{align*}Then using Jensen's inequality, we have
			\[
			\|\eps \p u(t,\cdot)\|_2^4\leq P(E_4(0))+\int_0^t P(E_4(\tau))\dtau.
			\]
			The other terms can also be controlled in the same way, as the control of $(\eps \p (\eps D_t)^m u)$ for $m\geq 3$ must generate an $\eps^2$ weight and we can put this weight into the term that does not have extra $\eps$ weight. What becomes different is that, the lowest-order term in the monomials might be $(\p(\eps D_t)^ju)$ for $j=0,1,2$. When $j=1,2$, we shall replace $L^\infty$ norm by $L^6$ norm and use $H^1\hookrightarrow L^6$. We take $(\eps \p (\eps D_t)^3 u)(\p(\eps D_t)^2 u)$ for an example. We have
			\begin{align*}
			&\|(\eps \p (\eps D_t)^3 u)(\p(\eps D_t)^2 u)\|_0^2\leq \|\eps^4\p D_t^3 u\|_{L^3}^2 \|\p (\eps D_t)^2 u\|_{L^6}^2 \leq \|\eps^4 D_t^3 u\|_2^2 \|(\eps D_t)^2 u\|_2^2\\
			\lesssim&~\eps^2\|(\eps D_t)^3 u(0,\cdot)\|_2^2\|(\eps D_t)^2 u\|_2^2 + \underbrace{\|(\eps D_t)^2 u\|_2^2\int_0^t \|(\eps D_t)^4 u(\tau,\cdot)\|_2^2 \dtau}_{\leq E_4(t)\int_0^t P(E_6(\tau))\dtau},\\
			\lesssim&~\|\eps (\eps D_t)^2 u\|_2^4 + \|(\eps D_t)^3 u(0,\cdot)\|_0^4 + E_4(t)\int_0^t P(E_6(\tau))\dtau
			\end{align*} and then apply the same argument to get
			\[
			\|\eps (\eps D_t)^2 u\|_2^4\leq P(E_4(0))+\int_0^t P(E_5(\tau))\dtau.
			\]
			We can also apply such method to all the other terms in the commutators. Therefore, we conclude that 					\begin{align}
			\left\|[\p, (\eps D_t)^{k+2l}]f\right\|_{3-k-l}^2\leq P(E(0)) + P(E(t))\int_0^t P(E(\tau))\dtau,
			\end{align}
			for $f=u,B,q,S$.
			\end{proof}
			
			\begin{lem}\label{lem Dt comm}
			Given integers $k,l$ satisfying $0\leq l\leq 3$ and $0\leq k\leq 3-l$, under the setting of Theorem \ref{main thm, ill data}, we have that
			\begin{align}
			\left\|[(\eps D_t)^{k+2l},g]f\right\|_{3-k-l}^2\leq P(E(0)) + P(E(t))\int_0^t P(E(\tau))\dtau,
			\end{align}
			for $(f,g)=(\eps D_t q, a), (\eps D_t u, \rho), (\eps(\nab\times B), B)$.
			\end{lem}
			\begin{proof}
			We only consider the estimates for $\left\|[(\eps D_t)^{k+2l},\eps B\times](\nab\times B)\right\|_{3-k-l}^2$ with $k=0, l=3$, which is also the most difficult case, in which it contains the highest number of material derivatives and $\nab\times B$ is a normal derivative instead of a tangential derivative. 
			
			In $[(\eps D_t)^6, \eps B\times](\nab\times B)$, the leading-order part gives us two terms (with certain coefficients)
			\[
			(\eps^6 D_t^6 B)\times(\eps\nab\times B)\text{ and }(\eps D_t B) \times (\eps^6D_t^5\nab\times B).
			\]
			These terms are controlled in the same way as in Lemma \ref{lem prod comm}. For example, the first term is controlled by
			\begin{align*}
			\|(\eps^6 D_t^6 B)\times(\eps\nab\times B)\|_0^2\lesssim \|\eps^6 D_t^6 B\|_0^2 \|\eps B\|_3^2,
			\end{align*}where $\|\eps^6 D_t^6 B\|_0^2$ has been controlled in Proposition \ref{Lem_TangentialEstimates} and Lemma \ref{lem transport} gives
			\[
			\|\eps B\|_3^2\lesssim \|\eps B_0\|_3^2 + \int_0^t \|\eps D_t B(\tau,\cdot)\|_3^2\dtau.
			\] The second term is controlled in a similar manner
			\begin{align*}
			&\left\|(\eps D_t B) \times (\eps^6D_t^5\nab\times B)\right\|_0^2\lesssim\|\eps D_t B\|_2^2\|\eps^6D_t^5\nab\times B\|_0^2\\
			\lesssim&~\|\eps D_t B\|_2^2\left(\eps^2\|\eps^5D_t^5\nab\times B(0,\cdot)\|_0^2+\int_0^t\|(\eps D_t)^6\nab\times B(\tau,\cdot)\|_0^2\dtau\right)\\
			\lesssim&~\|\eps^2 D_t B\|_2^4+\|\eps^5D_t^5\nab\times B(0,\cdot)\|_0^4+\underbrace{\|\eps D_t B\|_2^2\int_0^t\|(\eps D_t)^6\nab\times B(\tau,\cdot)\|_0^2\dtau}_{\leq E_4(t)\int_0^t P(E_7(\tau))\dtau}\\
			\lesssim&~P(E_4(0))+\int_0^t P(E_4(\tau))\dtau + P(E_6(0)) + E_4(t)\int_0^t P(E_7(\tau))\dtau.
			\end{align*}
			The other terms can be controlled in the same way and we skip the details.
			\end{proof}

		\end{appendix}

	\end{document}